\documentclass[]{article}

\usepackage[utf8]{inputenc}
\usepackage{tikz-cd}
\usepackage{amsmath}
\usepackage{enumerate}
\usepackage{amssymb}
\usepackage{enumitem}
\usepackage{geometry}
\usepackage{tikz}
\usepackage{mathtools}
\usetikzlibrary{decorations.pathmorphing}
\usepackage{bm}
\usepackage{extarrows}
\usepackage{stackengine,scalerel,graphicx,trimclip,wasysym}
\usepackage{graphicx}
\usepackage{amsthm}
\usepackage{ mathdots }
\usepackage{imakeidx}
\usepackage[colorlinks=false,pdfborder={0 0 1},urlbordercolor=red]{hyperref}
\usepackage{siunitx}

\usepackage{titlesec}

\titleformat{\section}
{\centering\normalfont\Large\bfseries} 
{\thesection}{1em}{}                   

\titleformat{\subsubsection}[block]
{\centering\normalfont\normalsize\bfseries} 
{\thesubsubsection}{1em}{}                  

\newtheorem{theorem}{Theorem}[section]
\newtheorem{proposition}[theorem]{Proposition}
\newtheorem{lemma}[theorem]{Lemma}
\newtheorem{corollary}[theorem]{Corollary}

\theoremstyle{definition}

\newtheorem{example}[theorem]{Example}
\newtheorem{examples}[theorem]{Examples}
\newtheorem{nonexample}[theorem]{Non-Example}
\newtheorem{definition}[theorem]{Definition}

\newtheorem{remark}[theorem]{Remark}

\def\A{\mathcal{A}}
\def\B{\mathcal{B}}
\def\C{\mathcal{C}}

\def\R{\mathbb{R}}
\def\Z{\mathbb{Z}}
\newcommand{\floor}[1]{\left\lfloor #1 \right\rfloor}
\newcommand{\ceil}[1]{\left\lceil #1 \right\rceil}

\def\ends{\mathcal{E}\text{nds}}
\newcommand{\suchthat}{\;\ifnum\currentgrouptype=16 \middle\fi|\;}

\tikzcdset{squiggly/.default={pre length=5pt, post length=5pt}}

\newcommand\rightarrowhead{\clipbox{4pt -2pt 0pt -2pt}{$\rightarrow$}}

\newcommand\rightsquigarrowbody{\clipbox{-2.5pt 0pt 3pt 0pt}{$\rightsquigarrow$}}

\newlength{\myrightsquigarrowshift}
\newcommand\myrightsquigarrow{
	\hspace{-\myrightsquigarrowshift}%
	\rightsquigarrowbody\rightarrowhead%
	\hspace{\myrightsquigarrowshift}%
}

\tikzset{
	Rightsquigarrow/.style={
		draw,
		>={Implies[]},
		->,
		double distance between line centers=1.5pt,
		decorate,
		decoration={
			zigzag,
			amplitude=0.7pt,
			segment length=3pt,
			pre length=5pt,
			post length=5pt
		}
	}
}

\title{Interactions between Coarse Homotopy and Ends on Proper Geodesic Spaces}
\author{Bradley Ashley}

\begin{document}

\maketitle

\begin{abstract}
	We consider the coarse-geometric notion of ends in the context of coarse homotopy. We show that, when recontextualized as a functor from an appropriate coarse category of proper geodesic spaces, the set of ends $\ends(-)$ is a coarse homotopy invariant. Further, we prove the existence of a natural surjection from the coarse path component functor $\pi_0^{\text{Crs}}(-)$ to $\ends(-)$, and show that in general, this is not an injection (even when restricted to locally finite planar graphs). Finally, we begin to consider when this injection indeed exists by showing that this is the case for locally finite geometric trees, providing a number of useful preliminary lemmas on the behaviour of geodesics in this context.	
\end{abstract}

\section{Introduction}
    
    The set of ends of a metric space, introduced by Freudenthal \cite{freudenthal1930enden}, is a tool in the study of asymptotic geometry. This aims to count the number of ways an observer may walk off to infinity in a given space, up to unbounded continuous path-connectedness. The cardinality of the set of ends of a given space is a well-known quasi-isometry invariant (see Proposition 8.29 of \cite{bridson2013metric} for example). A common example given is that $\R^n$ has two ends for $n=1$ and one end for $n\geq 2$.
    
    This invariant has many applications, most notably in the field of geometric group theory. For example the combined results of Hopf \cite{hopf1943enden} and Stallings \cite{stallings1968torsion} tell us that the Cayley graph of a finitely-generated group, viewed as a metric space via the canonical path metric, can have only zero, one, two,  or uncountably many ends, and further that there are concrete algebraic classifiers for groups with $0,1$ or infinitely many ends. When in the view of the Švarc–Milnor Lemma (see \cite{bridson2013metric} for details), we get a useful obstruction for the existence of a particularly nice form of group action on a metric space for groups whose number of ends lies outside these cardinalities.
    
    The set of coarse path components is a new idea, introduced by Mitchener, Norouzizadeh, and Schick \cite{mitchener2020coarse} where they develop a coarse homotopy theory. This is a tool which aims to count a space's coarsely connected components at infinity through the construction of points and paths `at infinity', a value invariant under their given definition of coarse homotopy equivalence. Among other examples, they show that $\R^n$ has two coarse path components for $n=1$, and one coarse path component for $n\geq 2$.
    
    There is intuitively some overlap here where in principle, both invariants appear to be designed to assess the same property of a given space, i.e. the number of connected infinities. In this paper, we explore this connection.
    
    We begin with an overview of large-scale geometry, including the equivalences of interest in quasi-isometries and coarse equivalences, noting the important example of quasi-geodesic spaces where these two notions align.
    
    This is followed by a section recalling the definition of the set of ends of a space. We provide some results and recall others useful for their computation, specifically when we restrict to the world of proper geodesic metric spaces. We recontextualise the set of ends as a functorial construction from the category of proper geodesic space and coarse maps $\textbf{PGeo}^{\text{Crs}}$ to the category of sets, and provide a proof showing that this `ends functor' maps large-scale equivalent morphisms to equivalent functions on respective sets of ends.
    
    In the proceeding two sections we cover Roe's definition of abstract coarse spaces and coarse maps and then Mitchener, Norouzizadeh, and Schick's notion of coarse homotopy \cite{mitchener2020coarse}. We make explicit some likely known but often left implicit results in these contexts that are useful for the work that follows.
    
    The sixth section covers the notion of coarse path components, also introduced in \cite{mitchener2020coarse}. We expand on their work by showing that the two canonical formulations define the same invariant. These formulations are analogous to those in the topological world where $\pi_0(X)$, for some space $X$, may be defined as the set $X$ with path-connected points identified, or equivalently as the set of homotopy classes of maps $*\to X$. In our context, points are coarse rays, and paths are mappings of the form $c_1([0,1]) \to X$, where $c_1(-)$ is the truncated metric cone defined in this section. We construct this coarse homotopy invariant functorially. Further, we show that for a proper geodesic metric space, every coarse path component may be represented by a rooted Lipschitz ray.
    
   	The following section is where the main work begins. We first show that for the metric cone of an appropriate space there is a canonical bijection between the set of ends and set of coarse path components. We then also show that there is a natural surjection from the `coarse path components functor $\pi_0^{\text{Crs}}(-)$ to the `ends functor' $\ends(-)$ when restricted to the category of proper geodesic metric spaces and coarse maps $\textbf{PGeo}^{\text{Crs}}$. An immediate corollary is then that, in this context, the notion of ends is a coarse homotopy invariant in the functorial sense. 
	
	With this natural surjection, it is then natural to ask whether or not it is also an injection, i.e., whether the invariants of coarse path components and ends align. We compute an example showing that this is not the case for general proper geodesic metric spaces (and in fact further for locally finite planar graphs equipped with the path metric). We would then like to begin considering when this injection might in fact exist. The remainder of the paper is devoted to showing that this is the case for locally finite geometric trees.
	
	Section 8 contains a small interlude providing some technical lemmas relating to the behaviour of geodesics on locally finite geometric trees (viewed as $1$-complexes equipped with the canonical path metric), where then, in the ninth and final section, we use these results to show that there is a natural injection (and hence isomorphism) from $\pi_0^{\text{Crs}}(-)$ to $\ends(-)$ when restricted to the class of locally finite geometric trees.


\section{Large-Scale Geometry}

In this section we cover some large-scale geometry preliminaries which we will need for the work that follows. We will denote a metric space by the pair $(X,d_X)$ except in some cases where the distance function has its own specific notation, e.g., $|-|$ for the standard metric on $\R_{\geq 0}$. When the context is clear we will just write $X$. For any point $x\in X$ and $R>0$ we will denote the open ball centred at $x$ of radius $R$ by $OB(x,R)$, i.e.,
\begin{equation*}
	OB(x,R) \coloneq \{x'\in X\ |\ d_x(x,x')<R\},
\end{equation*}
and similarly denote the closed ball centred at $x$ of radius $R$ by
\begin{equation*}
	CB(x,R) \coloneq \{x'\in X\ |\ d_x(x,x')\leq R\}.
\end{equation*}
Further, if $B\subseteq X$ is a bounded subset, we will denote its diameter by diam$(B)$, that is
\begin{equation*}
	\text{diam}(B) = \sup_{x,x'\in B}d_X(x,x').
\end{equation*}
The following is a common definition (see 2.5.1 and 2.5.2 of \cite{dructu2018geometric} for example).

\begin{definition}
	Let $X$ and $Y$ be metric spaces. A function $f:X \to Y$ is 
	\begin{enumerate}
		\item \textbf{$A$-Lipschitz} if there exists some $A\geq 0$ such that
		\begin{equation*}
			d_Y(f(x),f(x')) \leq A d_X(x,x'),
		\end{equation*}
		for all $x,x' \in X$.
		\item an \textbf{$A$-biLipschitz embedding} if there exists some $A\geq 1$ such that
		\begin{equation*}
			\frac{1}{A} d_X(x,x') \leq  d_Y(f(x),f(x')) \leq A d_X(x,x'),
		\end{equation*}
		for all $x,x' \in X$.
		\item an \textbf{$A$-biLipschitz equivalence} if $f$ is a surjective $A$-biLipschitz embedding.
	\end{enumerate}
\end{definition}

This is unfortunately too strong a definition to properly study large-scale structure. As such, we have the following common weakening (often termed a `coarsification') of the above definition.

\begin{definition}\label{quasi-iso-dfn}
	Let $X$ and $Y$ be metric spaces. Then a function $f:X \to Y$ is
	\begin{enumerate}
		\item \textbf{$(A,B)$-asymptotically Lipschitz} if there exists some $A,B \geq 0$ such that
		\begin{equation*}
			d_Y(f(x),f(x')) \leq A d_X(x,x') +B,
		\end{equation*}
		for all $x,x' \in X$.
		\item an \textbf{$(A,B)$-quasi-isometric embedding} if there exists some $A\geq 1$ and $B \geq 0$ such that
		\begin{equation*}
			\frac{1}{A}d_X(x,x') -B \leq d_Y(f(x),f(x')) \leq A d_X(x,x') +B,
		\end{equation*}
		for all $x,x' \in X$.
		\item a \textbf{$C$-quasi-surjection} if there exists some $C \geq 0$ such that for each $y\in Y$ we have
		\begin{equation*}
			d_Y(f(x),y) \leq C,
		\end{equation*}
		for some $x\in X$.
		\item an \textbf{$(A,B,C)$-quasi-isometry} if $f$ is both an $(A,B)$-quasi-isometric embedding and a $C$-quasi-surjection.
	\end{enumerate}
\end{definition}

We will say a map $f:X\to Y$ is \textbf{asymptotically Lipschitz} if $f$ is $(A,B)$-asymptotically Lipschitz for some $A,B$, a \textbf{quasi-isometric embedding} if $f$ is a $(A,B)$-quasi-isometric embedding for some $A,B$, a \textbf{quasi-surjection} if $f$ is a $C$-quasi-surjection for some $C$, and a \textbf{quasi-isometry} if $f$ is a $(A,B,C)$-quasi-isometry for some $A,B,C$. We similarly drop the constant $A$ for Lipschitz maps and biLipschitz embeddings and equivalences.

\begin{example}\label{crs_maps_basic_examples}
	For each $a\in \R$ the map $\R \to \R$ defined by $r \mapsto ar$ is asymptotically Lipschitz, and proper if $a\neq 0$. Perhaps more interestingly, the floor function $\floor{-}:\R \to \Z$ is an example of a quasi-isometry that is not continuous.
\end{example}

Asymptotically Lipschitz maps and quasi-isometries are key in the study of large-scale geometry, and more specifically geometric group theory (see chapters 8 and 9 of \cite{dructu2018geometric} for example). They are not, however, the only sort of `large-scale' morphisms studied. In this paper we will be more concerned with the coarse equivalence.

\begin{definition}\label{metric_coarse_map_dfn}
	Let $(X,d_X),(Y,d_Y)$ be metric spaces. Then a map $f: X \to Y$ is
	\begin{enumerate}
		\item \textbf{controlled} if for all $R>0$ there exists $S(R)>0$ such that for all $x,x' \in X$
		$$d_X(x,x')<R$$
		implies
		$$d_Y(f(x),f(x'))<S(R),$$
		\item \textbf{metrically proper} if for any bounded set $B\subseteq Y$, $f^{-1}(B)$ is bounded in $X$, and
		\item \textbf{coarse} if both controlled and metrically proper.
	\end{enumerate}
\end{definition}

We will often refer to a metrically proper map as just \textbf{proper}. The exception to this will be when we are considering \textbf{topologically proper} maps which are defined the same as above except with compact subsets rather than bounded. More details on this are in the upcoming subsection on Ends. 

\begin{example}
	The map $\R^2 \to \R$ defined by $(r,r')\mapsto r+r'$ is controlled, and coarse when restricted to $\R_{\geq 0}^2$.
\end{example}

\begin{nonexample}
	The map $\R_{\geq 0}\to \R_{\geq 0}$ defined by $x\mapsto x^2$, for all $x$, fails to be controlled, and the constant map $\R_{\geq 0} \to \R_{\geq 0}$ defined by $x \mapsto 0$, for all $x$, fails to be proper.
\end{nonexample}

\begin{definition}\label{close_metric}
	Let $X,Y$ be metric spaces and let $f,g: X \to Y$ be functions. Then \textbf{$f$ is close to $g$} if there exists some $C \geq 0$ such that 
	\begin{equation*}
		d_Y(f(x),g(x)) \leq C
	\end{equation*}
	for all $x\in X$.
\end{definition}

\begin{definition}\label{metric_crs_equiv_dfn}
	Let $X$ and $Y$ be metric spaces and $f: X \to Y$ a controlled map. Then $f$ is a \textbf{coarse equivalence} if there exists a controlled map $g:Y\to X$ such that $gf$ is close to $1_{X}$ and $fg$ is close to $1_Y$, denoted $X \overset{\text{Crs}}{\simeq} Y$.
\end{definition}

\begin{remark}\label{crdeqiscrseq}
	It was shown in \cite{mohamad2013coarse} Proposition 1.4.4 that any coarse equivalence is necessarily proper.
\end{remark}

\begin{example}
	Let $\Gamma$ be a finitely-generated group and let $S$ be a choice of finite generating set. Recall the definitions of the word metric space on $\Gamma$ with respect to $S$ denoted $|\Gamma|_S$, and the Cayley graph of $\Gamma$ with respect to $S$ denoted $\text{Cay}(\Gamma,S)$, equipped with the canonical path metric. Both of these constructions can be found in \cite{dructu2018geometric}.
	
	Suppose $T$ is another choice of generating set for $\Gamma$. Then there is are coarse equivalences between metric spaces
	\begin{equation*}
		\text{Cay}(\Gamma,S) \overset{\text{Crs}}{\simeq} |\Gamma|_S \overset{\text{Crs}}{\simeq} |\Gamma|_T \overset{\text{Crs}}{\simeq} \text{Cay}(\Gamma,T).
	\end{equation*}
\end{example}

The above example is not surprising, it is well known that there are quasi-isometries above \cite{dructu2018geometric}. It turns out that in many cases, the notions of coarse equivalence and quasi-isometry are the same. This can be useful as it can often be easier to prove the existence of the former.

\begin{definition}
	Let $X$ be a metric space and $x,x'$ be points in $X$. Then a \textbf{geodesic segment from $x$ to $x'$}, or just a \textbf{geodesic from $x$ to $x'$}, is an isometric embedding $u_{x,x'}:[0,a_{x,x'}] \to X$ such that $u_{x,x'}(0) = x$, $u_{x,x'}(a_{x,x'}) = x'$ and $a_{x,x'}=d_X(x,x')$.
	
	We call the a $X$ \textbf{geodesic} if any pair of points in $X$ may be joined by a geodesic. Further, $X$ is called \textbf{uniquely geodesic} if there is precisely one geodesic connecting any pair of points.
\end{definition}

\begin{proposition}
	Let $X$ be a geodesic metric space, and $x,x',x''$ be points in $X$ such that we have $x'\in \text{im}(u_{x,x''})$. Then the map $u_{x,x''}:[0,(u_{x,x''})^{-1}(x')] \to X$ is a geodesic from $x$ to $x'$, and the map $u_{x,x''}:[(u_{x,x''})^{-1}(x'),a_{x,x''}] \to X$ is a geodesic from $x'$ to $x''$.
\end{proposition}

\begin{proof}
	Straightforward.
\end{proof}

Naturally, the notion of geodesics has its own \textit{coarsification}.

\begin{definition}
	Let $X$ be a metric space and $x,x'$ be points in $X$. Then a \textbf{$(\lambda,\epsilon)$-quasi-geodesic segment from $x$ to $x'$}, or just a \textbf{quasi-geodesic from $x$ to $x'$}, is a $(\lambda,\epsilon)$-quasi-isometric embedding $u_{x,x'}^q : \big[0,a^q_{x,x'}\big] \to X$ for some $a^q_{x,x'}\geq 0$, such that $u_{x,x'}^q(0) = x$ and $u_{x,x'}^q\big(a_{x,x'}^q\big) = x'$.
	
	We call a metric space $X$ \textbf{quasi-geodesic} if there exists constants $\lambda \geq 1$ and $\epsilon \geq 0$ such that any pair of points in $X$ may be joined by a quasi-geodesic.
\end{definition}

\begin{proposition}
	Every geodesic is a quasi-geodesic, and therefore every geodesic space is a quasi-geodesic space.
\end{proposition}

\begin{proof}
	Simply set $\lambda =1, \epsilon = 0$, and $a_{x,x'} = d_X(x,x')$, and we are done.
\end{proof}

We will show that for the class of quasi-geodesic spaces, the notions of quasi-isometry and coarse equivalence align.

\begin{proposition}[Proposition 1.3.11, \cite{mohamad2013coarse}]\label{qiiffalcrs}
	Let $X$ and $Y$ be metric spaces. Then a function $f:X \to Y$ is a quasi-isometry if and only if $f$ is asymptotically Lipschitz and there exists another asymptotically Lipschitz map $g: Y \to X$ such that $gf$ is close to $1_X$ and $fg$ is close to $1_Y$. 
\end{proposition}

The following proposition and corollary can be found in \cite{nowak2012large}.

\begin{proposition}\label{aL_implies_crd}
	Let $X$ and $Y$ be metric spaces. Then any asymptotically Lipschitz $f:X \to Y$ is controlled. Further, if $X$ is a quasi-geodesic space, the converse holds also.
\end{proposition}

\begin{corollary}\label{qiiffcrseq}
	Let $X$ and $Y$ metric spaces, and $f:X \to Y$ be a function. Then if $f$ is a quasi-isometry then $f$ is a coarse equivalence. Further, if $X$ and $Y$ are quasi-geodesic spaces, the converse holds also.
\end{corollary}

\begin{definition}
	Let $X$ be a metric space and $r>0$. Then $X$ is $r$-discrete if for all distinct elements $x,x' \in X$ we have $d_X(x,x')\geq r$.
\end{definition}

\begin{proposition}\label{discretealislip}
	Let $X$ and $Y$ be metric spaces such that $X$ is $r$-discrete for some $r>0$. Then any asymptotically Lipschitz map $f:X \to Y$ is Lipschitz.
\end{proposition}

\begin{proof}
	For all distinct elements $x,x' \in X$, we have
	\begin{align*}
		d_X(x,x')\geq r \iff \frac{1}{r} d_X(x,x') \geq 1.
	\end{align*}
	Suppose $f$ is $(A,B)$-asymptotically Lipschitz. Then for each pair of distinct elements $x,x' \in X$ we have
	\begin{align*}
		d_Y(f(x),f(x')) & \leq A d_X(x,x') + B\\
		& \leq Ad_X(x,x') + B \frac{1}{r} d_X(x,x')\\
		& = \left(A + \frac{B}{r}\right)d_X(x,x'),
	\end{align*} 
	and so $f$ is $\left(A + \frac{B}{r}\right)$-Lipschitz.
\end{proof}

\begin{corollary}\label{discrete_crs_is_lip}
	Let $X$ be a quasi-geodesic metric space and $Y$ a metric space, and $A\subseteq X$ a $r$-discrete subset for some $r>0$. Then for any coarse map $f:X \to Y$ we have that $f|_A$ is Lipschitz.
\end{corollary}

\begin{proof}
	We have that $f$, and therefore $f|_A$, is asymptotically Lipschitz by Proposition \ref{aL_implies_crd} and therefore $f|_A$ is Lipschitz by Proposition \ref{discretealislip}.
\end{proof}


\section{Ends}

In this section, we recall the definition of the set of ends of a space, a tool useful in the study of large-scale geometry, notably in the field of geometric group theory \cite{dructu2018geometric}. Here, we count the number of distinct infinities of a given space. In order to give the definition, we first need to consider different but related notion of a proper function.

\begin{definition}
	Let $X$ and $Y$ be metric spaces. Then we will call a map $f:X \to Y$ \textbf{topologically proper} if for all compact $K\subseteq Y$ we have that $f^{-1}[K]$ is compact in $X$.
\end{definition}

For a metric space $X$ we will often refer to function $r:\R_{\geq 0} \to X$ of metric spaces as a \textbf{ray} or a \textbf{ray in $X$}, most often with preceding adjectives. The key adjectives being coarse, topologically proper, metrically proper, continuous, Lipschitz, asymptotically Lipschitz, and geodesic, where the latter is defined to be an isometric embedding of $\R_{\geq 0}$. We will also call a ray $x_0$\textbf{-rooted}, for some $x_0\in X$, if $r(0)=x_0$. If the context is clear, we will often refer to a ray $r:\R_{\geq 0} \to X$ by $r$.

The following is a common definition.

\begin{definition}[Definition 8.27, \cite{bridson2013metric}]
	Let $X$ be a metric space. Let $r,r':\R_{\geq 0} \to X$ be topologically proper continuous rays. We will declare them equivalent if for all compact $K\subseteq X$ there exists some $t\geq 0$ such that $\text{im}\big(r|_{[t,\infty)}\big)$ and $\text{im}\big(r'|_{[t,\infty)}\big)$ lie in the same path component of $X\backslash K$.
	
	This defines an equivalence relation on the set of topologically proper continuous rays in $X$, called the \textbf{set of ends of $X$}, denoted $\ends(X)$, where the class represented by $r:\R_{\geq 0} \to X$ is called an \textbf{end of $X$} and is denoted $\text{end}(r)$.
\end{definition}

This isn't quite the definition that we want to use, however. In particular, we want bounded sets and compact sets to play the same role. We can do this if we restrict to proper metric spaces.

\begin{definition}\label{proper_space_dfn}
	A metric space $X$ is called \textbf{proper} if for every subset $B\subseteq X$ the following are equivalent.
	\begin{enumerate}
		\item $B$ is both closed and bounded.
		\item $B$ is compact.
	\end{enumerate}
\end{definition}

\begin{example}
	The proto-example of a proper metric space is $\R^n$ for any $n\in \Z_{\geq 1}$. This property is shown in the well-known Heine-Borel Theorem (see \cite{royden1988real} for instance). This then extends to any closed subspace $X \subseteq \R^n$, since a subset is closed and bounded in $X$ if and only if it is closed and bounded in $\R^n$.
	
	Another important and well-known example is locally finite geometric graphs (see Example 2.A.13 of \cite{cornulier2014metric} for instance), that is graphs equipped with the canonical path-metric. This will be important when we go on to look at locally finite geometric trees (which can be viewed as contractible locally finite geometric graphs).
\end{example}

\begin{proposition}\label{metric_proper_is_top_proper}
	Let $X$ and $Y$ be a proper metric space and $f:X \to Y$ be a continuous map, then the following are equivalent.
	\begin{enumerate}
		\item The map $f$ is topologically proper. That is, for all compact sets $K\subseteq Y$ we have that $f^{-1}(K)$ is compact in $X$.
		\item The map $f$ is metrically proper. That is, for all bounded sets $B\subseteq Y$ we have that $f^{-1}(B)$ is compact in $X$.
	\end{enumerate}
\end{proposition}

\begin{proof}
	Assume $f$ is topologically proper and let $B \subseteq Y$ be bounded with diam$(B)=R$. Let $b \in B$, then $B\subseteq CB(b,R)$. Then, $CB(b,R)$ is compact in $Y$ since $Y$ is proper, and so
	\begin{align*}
		f^{-1}(B) \subseteq f^{-1}(CB(b,R)),
	\end{align*}
	is compact in $X$ by topological properness, and hence bounded.
	
	Now, assume $f$ is metrically proper and let $K \subseteq Y$ be compact. Then $K$ is closed and bounded. Then since $f$ is continuous and metrically proper, $f^{-1}(B)$ is closed and bounded in $X$, and hence compact by the properness of $X$.
\end{proof}

By Proposition \ref{metric_proper_is_top_proper} and the fact that in a proper geodesic space every compact set is bounded and every bounded set is contained in a compact set (a closed ball of greater diameter to be specific), if we restrict to the subclass of proper metric spaces the following definition for the set of ends is equivalent to the one given above.

\begin{definition}\label{ends_proper_space_dfn}
	Let $X$ be a proper space. Let $r,r':\R_{\geq 0} \to X$ be metrically proper continuous rays. We will declare them equivalent if for all bounded $B\subseteq X$ there exists some $t\geq 0$ such that $\text{im}(r|_{[t,\infty)})$ and $\text{im}(r'|_{[t,\infty)})$ lie in the same path component of $X\backslash B$.
	
	This defines an equivalence relation on the set of metrically proper continuous rays in $X$, called the \textbf{set of ends of $X$}, denoted $\ends(X)$, where the class represented by $r:\R_{\geq 0} \to X$ is called an \textbf{end of $X$} and is denoted $\text{end}(r)$.
\end{definition}

The following is a useful lemma.

\begin{lemma}[Lemma 8.28, \cite{bridson2013metric}]\label{bridsonendslemma}
	Let $X$ be a proper geodesic space, $x_0$ be some point in $X$. Then the following hold.
	\begin{enumerate}
		\item Let $r,r':\R_{\geq 0} \to X$ be proper continuous rays in $X$. Then end$(r)=$end$(r')$ if and only if for all $R>0$ there exists $T>0$ such that there exists some $k$-path path from $r(t)$ to $r'(t)$ contained in $X\backslash CB(x_0,R)$, for some $k>0$ and for all $t>T$.
		\item Each class in $\ends(X)$ may be represented by some geodesic ray.
	\end{enumerate}
\end{lemma}

The ideas below are essentially an adaptation of Proposition 8.29 of \cite{bridson2013metric}. Here, we make the effort to use categorical language. This will make these ideas easier to talk about in an upcoming section where we wish consider the set of ends in the context of coarse homotopy on the categorical level.

\begin{definition}\label{continuousray}
	Let $X$ be a geodesic space and $r: \R_{\geq 0} \to X$ be any map. Then we define the continuous map $r^*$ as follows
	\begin{align*}
		r^*(h) \coloneqq u_{r(\lfloor h\rfloor),r(\lceil h \rceil)}\left(a_{r(\lfloor h\rfloor),r(\lceil h \rceil)}(h-\lfloor h \rfloor)\right)
	\end{align*}
	where $u_{r(\lfloor h\rfloor),r(\lceil h \rceil)} : \left[0,a_{r(\lfloor h\rfloor)}\right] \to X$ is a choice of geodesic segment from $r(\lfloor h \rfloor)$ to $r(\lceil h \rceil)$. In particular, $r^*|_{\Z_{\geq 0}}=r|_{\Z_{\geq 0}}$.
\end{definition}

Obviously, the above definition is not a well-defined one in general (specifically if $X$ is not uniquely geodesic), however it is defined as well as we need in the context of ends by the following proposition.

\begin{proposition}
	Let $X$ be a proper geodesic space and $r: \R_{\geq 0} \to X$ be a proper continuous ray. Then $r^*$ is proper and end$(r)=$end$(r^*)$ where any choice of geodesic segments are used to define $r^*$.
\end{proposition}

\begin{proof}
	Straightforward.
\end{proof}

Let $\textbf{PGeo}^{\text{Crs}}$ be the category of proper geodesic metric spaces and a coarse maps.

\begin{theorem}\label{ends_functor}
	The following data defines a functor $\ends(-):\textbf{PGeo}^{\text{Crs}} \to \textbf{Set}$.
	\begin{itemize}
		\item $\ends(-): X \mapsto \ends(X)$ for each $X\in \textbf{PGeo}^{\text{Crs}}$, and
		\item $\ends(-) : (f:X \to Y) \mapsto (f_{\mathcal{E}}: \ends(X) \to \ends(Y))$ where
		\begin{equation*}
			f_{\mathcal{E}}: \text{end}(r) \mapsto \text{end}((f\circ r)^*),
		\end{equation*}
		for any morphism $f:X \to Y$.
	\end{itemize}
	Further, $\ends(-)$ sends close maps to equal maps.
\end{theorem}

\begin{proof}
	To see $f_{\mathcal{E}}$ is well-defined, by Lemma \ref{bridsonendslemma} it suffices to show that for any pair of $x_0$-rooted geodesic rays $r,r'$ in $X$, for some $x_0$, such that $\text{end}(r)=\text{end}(r')$ we have 
	\begin{align*}
		f_{\mathcal{E}}(\text{end}(r)) = \text{end}((f\circ r)^*) = \text{end}((f\circ r')^*) = f_{\mathcal{E}}(\text{end}(r')).
	\end{align*}
	Let $B$ be bounded in $Y$ and consider $f^{-1}(B)$, which is bounded in $X$ since $f$ is proper. Then, there exists some $R>0$ such that $f^{-1}(B)\subseteq OB(x_0,R)$. Then, by Lemma \ref{bridsonendslemma}, there exists some $k>0$ and $T>0$ such that there is a $k$-path, say
	\begin{equation*}
		r(t) = x_{0,t},x_{1,t},...,x_{n_t,t}=r'(t)
	\end{equation*}
	from $r(t)$ to $r'(t)$ for all $t>T$. Since $f$ is controlled, for all $x,x'\in X$, there exists $S(k)>0$ and $S(1)>0$ such that if $d_X(x,x')\leq k$ then $d_Y(f(x),f(x'))<S(k)$, and if $d_X(x,x')\leq 1$ then $d_Y(f(x),f(x'))<S(1)$. Let $k'\coloneq \text{max}(S(k),S(1))$.
	
	Then, there exists some $R'$ with $f[OB(x_0,R)]\subseteq OB(f(x_0),R')$. Further, since $r$ an $r'$ are proper, there exists some $T'>0$ with
	\begin{equation*}
		\text{im}\big(f\circ r|_{[T',\infty)}\big)\cap OB(f(x_0),R') = \emptyset,
	\end{equation*}
	and
	\begin{equation*}
		\text{im}\big(f\circ r'|_{[T',\infty)}\big)\cap OB(f(x_0),R') = \emptyset.
	\end{equation*}
	Then, for all $t>\text{max}(T,T')$, there is a $k'$-path from $f\circ r(t)$ to $f\circ r'(t)$, given by
	\begin{equation*}
		f\circ r(t), f\circ r(\ceil{t}) = f\big(x_{0,\ceil{t}}\big),f\big(x_{1,\ceil{t}}\big),...,f\big(x_{n_{\ceil{t}},\ceil{t}}\big)=f\circ r'(\ceil{t}), f\circ r'(t),
	\end{equation*}
	and $f_{\mathcal{E}}$ is well-defined.
	
	To see functionality, let $f:X \to Y$ and $g:Y \to Z$ be morphisms in $\textbf{PGeo}^{\text{Crs}}$ and let $\text{end}(r)\in \ends(X)$. Then we claim 
	\begin{equation*}
		\text{end}\left(\left((g\circ f)\circ r \right)^*\right) = \text{end}\left(\left(g\circ (f\circ r)^* \right)^*\right).
	\end{equation*}
	To see this, let $B$ be bounded in $Y$. Then since both $\left((g\circ f)\circ r \right)^*$ and $\left(g\circ (f\circ r)^* \right)^*$ are proper, there exists $R>0$ such that
	\begin{equation*}
		\text{im}\left(\left((g\circ f)\circ r \right)^*|_{[R,\infty)}\right) \cap B = \emptyset,
	\end{equation*}
	and
	\begin{equation*}
		\text{im}\left(\left(g\circ (f\circ r)^* \right)^*|_{[R,\infty)}\right) \cap B = \emptyset.
	\end{equation*}
	Then just pick $\ceil{R}$, then
	\begin{equation*}
		\left((g\circ f)\circ r \right)^*(\ceil{R}) = \left(g\circ (f\circ r)^* \right)^*(\ceil{R})
	\end{equation*}
	and $\ends(-)$ is a functor.
	
	Finally, suppose we have another map $f':X \to Y$ such that $f$ is close $f'$. Again, by Lemma \ref{bridsonendslemma} is suffices to show that for all $x_0$-rooted geodesic rays $r$ in $X$ we have $f_{\mathcal{E}}(r)=f'_{\mathcal{E}}(r)$. There exists some $C\geq 0$ such that $d_Y(f(x),f'(x))\leq C$. Similar to above, since $f'$ is controlled, for all $x,x'\in X$, if $d_X(x,x')<1$ then $d_Y(f'(x),f'(x'))<S'(1)$. Then let $k''\coloneq \text{max}(S(1),S'(1),C)$, for $S(1)$ above, and let $R>0$.
	
	Then, since $f$ and $f'$ are proper, there exists some $T''>0$ such that
	\begin{equation*}
		\text{im}\big(f\circ r|_{[T'',\infty)}\big)\cap OB(f(0),R') = \emptyset,
	\end{equation*}
	and
	\begin{equation*}
		\text{im}\big(f\circ r|_{[T'',\infty)}\big)\cap OB(f(0),R') = \emptyset.
	\end{equation*}
	Where we make the arbitrary choice $f(0)$ for the centre of the ball. Then, for all $t>T''$ there is a $k''$-path from $f\circ r(t)$ to $f'\circ r(t)$, given by
	\begin{equation*}
		f\circ r(t),f\circ r(\ceil{t}),f'\circ r(\ceil{t}), f'\circ r(t),
	\end{equation*}
	and we are done.
\end{proof}

\begin{corollary}\label{ends_quasi_inv}
	The functor $\ends(-):\textbf{PGeo}^{\text{Crs}} \to \textbf{Set}$ maps coarse equivalences (and hence quasi-isometries) to bijections.
\end{corollary}

\begin{proof}
	Follows from Theorem \ref{ends_functor}.
\end{proof}

The following is also an important result in the field.

\begin{theorem}[Theorem 8.32, \cite{bridson2013metric}]\label{ends_of_groups}
	Let $\Gamma$ be a finitely-generated group. Then
	\begin{enumerate}
		\item $|\ends(\Gamma)|\in\{0,1,2\}$ or is uncountably infinite.,
		\item $|\ends(\Gamma)| = 0$ if and only if $\Gamma$ is a finite group,
		\item $|\ends(\Gamma)| = 2$ if and only if $\Gamma$ is quasi-isometric to $\Z$, and
		\item $|\ends(\Gamma)|$ is uncountably infinite if and only if $\Gamma$ can be expressed as an amalgamated free product $A*_cB$ or HNN extension $A*_C$ for finite $C$, $|A\backslash C| \geq 3$ and $|B \backslash C|\geq 2$.
	\end{enumerate}
\end{theorem}

The authors of \cite{bridson2013metric} point out $(1)$-$(3)$ above are due to Hopf \cite{hopf1943enden} and $(4)$ is due to Stallings \cite{stallings1968torsion}. Definitions of an amalgamated free product of HNN extension are given in section $(III.\Gamma.6)$ of \cite{bridson2013metric}. 


\section{Abstract Coarse Geometry}

Much like how one may abstractify the notion of open sets as a topology, we can do a similar thing with large-scale structure. This provides an axiomatic language allowing for easier defining of concepts such as a large-scale notion of products. 

The following definition is now a well-known one. In this form it is adapted from the one given in \cite{mitchener2020coarse}.

\begin{definition}
	Consider a set $X$. A \textbf{coarse structure} on X, $\C_X$, is a collection of subsets of the cartesian product $X \times X$, called \textbf{controlled sets}, such that
	\begin{enumerate}
		\item the diagonal of $X$, $\Delta_X\coloneq \{(x,x)\ |\ x \in X\} \subseteq X \times X$, is a controlled set,
		\item for any controlled set $U\in\C_X$ and $V\subseteq U$ we have that $V$ is controlled,
		\item for any finite collection $(U_i)_{i \in \mathcal{I}}$ such that each $U_i$ is a controlled set, we have that the finite union $\bigcup_{i \in \mathcal{I}} U_i$ is also a controlled set,
		\item for any controlled set $U \in \C_X$, its \textbf{inverse} defined by
		\begin{align*}
			U^{-1} := \{(x,x')\in X\times X\ |\ (x',x) \in U\}
		\end{align*}
		is an controlled set, and
		\item for any two controlled sets $U,V \in \C_X$, their \textbf{composite}, defined as
		\begin{align*}
			U \circ V := \{(x,x') \in X \times X\ |\ \exists\ x'' \in X : (x,x'') \in U, (x'',x') \in V\}
		\end{align*}
		is an controlled set.
	\end{enumerate}
	A \textbf{coarse space} $(X,\C_X)$ is a set $X$ equipped with a coarse structure $\C_X$.
\end{definition}

We will often denote a coarse space $(X,\C_X)$ just by its set $X$ when the context is clear. If $U \in \C_X$ and $U=U^{-1}$ then $U$ is said to be a \textbf{symmetric controlled set}. Every controlled set $U \in \C_X$ is the subset of a symmetric controlled set, given by $U \cup U^{-1}$, which we denote as $\text{sym}(U)$ and call the \textbf{symmetrization of $U$}.

\begin{examples}\label{coarse_space_basic_examples}
	\begin{enumerate}
		\item The immediate, and certainly most important, example of a coarse structure is that generated by a metric. Let $(X,d)$ be a metric space. Then we can define the \textbf{metric coarse structure} $\C_X^d$ on $X$ where a subset $U \subseteq X \times X$ is controlled if $\text{sup}\{d(x,x')\ |\ (x,x')\in E\ \}< \infty$. That is if $U$ is a subset of the following set
		\begin{equation*}
			U^R_{X} \coloneq \{(x,x') \in X\times X\ |\ d_X(x,x')<R \},
		\end{equation*}
		for some $R>0$.
		\item Any set $X$ may be equipped with the \textbf{minimum coarse structure}, $\C_X^{\text{min}}$, which is defined to be the set of all finite unions of sets of the form $\{(x,x)\}$ for some $x\in X$, and the \textbf{maximum coarse structure}, $\C_X^{\text{max}}$, which is defined to be the set of all subsets of $X\times X$. We may denote any set $X$ equipped with the minimum or maximum coarse structure by $X^{\text{min}}$ or $X^{\text{max}}$ respectively.
	\end{enumerate}
\end{examples}

\begin{definition}
	Let $(X,\C_X)$ be a coarse space. A \textbf{bounded set} in $X$ is a subset $B\subseteq X$ such that $B\times B \in \C_X$. We will borrow some notation from \cite{bunke2020homotopy} and denote the collection of bounded sets in $X$ by $\B_X$. 	We call a coarse space $X$ \textbf{bounded} if and only if $X\times X \in \C_X$ (i.e., $X\in \B_X$).	
\end{definition}

\begin{example}
	A bounded set in a metric coarse structure is precisely a bounded set with respect to the metric. Further, a metric space $X$ itself is bounded if and only if $X\times X$ is controlled with respect to the metric coarse structure.
\end{example}

\begin{definition}
	Let $(X,\C_X),(Y,\C_Y)$ be coarse spaces and $f:X \to Y$ be a function. Then $f$ is
	\begin{enumerate}
		\item \textbf{controlled} if $(f\times f)(E) \in \C_Y$ for all $E \in \C_X$,
		\item \textbf{proper} if $f^{-1}(B)$ is bounded in $X$ for all $B$ bounded in $Y$, and
		\item \textbf{coarse} if both controlled and proper.
	\end{enumerate}
\end{definition}

A category we will often be interested in is made up of the class of coarse spaces and coarse maps which we will denote \textbf{Crs}.

\begin{example}\label{metric_crs_=_crs_crs}
	Let $X$ and $Y$ be metric spaces. Then a map $f:X \to Y$ is controlled (resp. proper, coarse) with respect to their respective metric coarse structures if and only if $f$ is controlled (resp. proper, coarse) with respect to Definition \ref{metric_coarse_map_dfn}.
\end{example}

We will slightly abuse terminology when referring to the metric coarse structure and often treat a metric space $X$ as both a coarse space and a metric space interchangeably. Further, we will view any category of metric spaces and coarse maps to be a subcategory of \textbf{Crs}.

\begin{definition}
	Let $(X,\C_X)$ be a coarse space and $Y\subseteq X$. Then we can define the \textbf{coarse subspace structure} $\C_Y^X$ on $Y$ to be the collection of subsets $E\subseteq Y\times Y$ where $E\in \C_Y^X$ if and only if $E \in \C_X$.
\end{definition}

\begin{definition}
	Let $(X,\C_X),(Y,\C_Y)$ be coarse spaces. Then we define the \textbf{product coarse structure} $\C_{X\times Y}$ on the set $X\times Y$ to be the collection of subsets $E\subseteq (X\times Y)\times (X\times Y)$ where $E \in \C_{X\times Y}$ if and only if $\pi_{X\times X}(E) \in \C_X$ and $\pi_{Y\times Y}(E) \in \C_Y$, where
	\begin{equation*}
		\pi_{X\times X}: (X\times Y) \times (X \times Y) \to X\times X
	\end{equation*}
	and so
	\begin{equation*}
		\pi_{Y\times Y}: (X\times Y) \times (X \times Y) \to Y\times Y
	\end{equation*}
	are the canonical projections.
\end{definition}

\begin{remark}
	It should be noted that the above definition does not give a product in the categorical sense in \textbf{Crs}. This is because the projections are not in general coarse maps (they are controlled but not proper). We do, however, have the following useful maps.
\end{remark}

\begin{proposition}\label{crs_maps_for_products}
	Let $X$ and $Y$ be coarse spaces. Fix some $y \in Y$, then the inclusion $X \hookrightarrow X\times Y$ with $x\mapsto (x,y)$ is coarse. Further, for any other pair of coarse spaces $X'$ and $Y'$ and coarse maps $f:X \to X'$ and $g:Y \to Y'$, the canonical map $f\times g : X\times Y \to X' \times Y'$ is coarse. 
\end{proposition}	

\begin{proof}
	Straightforward.
\end{proof}

An isomorphism in the category \textbf{Crs} is precisely a bijection for which the maps in both directions are coarse. We will term these \textbf{coarse isomorphisms}, and denote them with $\cong$. 

While coarse isomorphisms are useful for some technicalities, they are not in general the equivalence most useful for the study of large-scale structure. The following definition introduces a weaker but more geometrically interesting form of equivalence, which are the focal point for much of the literature, and align better with the equivalences we have seen thus far.

\begin{definition}\label{closedfn}
	Let $(X,\C_X),(Y,\C_Y)$ be coarse spaces and $f,g:X \to Y$ be coarse maps. Then we say $f$ and $g$ are \textbf{close} if $(f\times g)(\Delta_X) \in \C_Y$.
	
	Further, we say $f$ is a \textbf{coarse equivalence} if there exists a coarse map $f':Y \to X$, termed \textbf{the coarse inverse of $f$}, such that $g\circ f$ is close to $ 1_{X}$ and $f\circ g$ is close to $1_Y$.
\end{definition}

\begin{example}\label{abstract_crs_equiv_is_metric_crs_equiv}
	A pair of coarse maps in the metric sense (Definition \ref{close_metric}) are close if and only if they are close with respect to the metric coarse structures (Definition \ref{closedfn}). Further, any coarse map between metric spaces is a coarse equivalence in the metric sense (Definition \ref{metric_crs_equiv_dfn}) if and only if it is a coarse equivalence with respect to the metric coarse structures (Definition \ref{closedfn}).
\end{example}

The following is useful for the explicit defining of coarse maps, in particular piecewise ones.

\begin{proposition}
	Let $X$ and $Y$ be coarse spaces such that $\B_X$ is closed under finite unions, and $f:X\to Y$ be a well-defined piecewise function with a finite number of pieces. Then $f$ is proper if and only if each piece is proper.
\end{proposition}

\begin{proof}
	Let $f$ a well-defined piecewise function with $n$ pieces for some $b$, i.e we have $f(x)=f_i(x)$ for some $i \in \boldsymbol{n}$, subset $A_i \subseteq X$, and function $f_i:A_i \to Y$. The forwards implication is straightforward, so for the sake of the backwards implication let $B\in\B_Y$. Then, if each $f_i$ is proper, we have
	\begin{equation*}
		f^{-1}[B]=\bigcup_{i \in\boldsymbol{n}}f^{-1}_i[B] \in \B_X,
	\end{equation*}
	and $f$ is proper.
\end{proof}

\begin{remark}
	Note, not all coarse spaces $X$ have the property that $\B_X$ is closed under finite unions, for example this is the case for the minimum coarse structure on some set $X$, $\C_X^{\text{min}}$. This does, however, include the class of all metric spaces.
\end{remark}

Checking whether a piecewise function is controlled is less straightforward. We fortunately have the following definition and proceeding lemma providing a criteria to make this job easier in many cases.

\begin{definition}
	Let $X$ be a coarse space and $A,B \subseteq X$ such that $X=A \cup B$. Then we say the decomposition $X=A \cup B$ \textbf{coarsely excisive} if for each $U \in \C_X$ there exists $U' \in \C_X$ such that $U[A]\cap U[B] \subseteq U'[A \cap B]$.
\end{definition}

\begin{lemma}[Lemma 2.2.7, \cite{mohamad2013coarse}]
	Let $X$ and $Y$ be coarse spaces and $X=A\cup B$ be a coarsely excisive decomposition. Suppose $f:X \to Y$ is a function. Then if $f|_A$ and $f|_B$ are coarse maps if and only if $f$ is a coarse map.
\end{lemma}

\begin{proposition}[Corollary 2.2.8, \cite{mohamad2013coarse}]\label{crs_excisive_for_R}
	Let $X$ be a coarse space and let $q:X \to \R_{\geq 0}$ be some controlled map. Let
	\begin{equation*}
		A = \{(x,t)\in X \times \R_{\geq 0}\ |\ t \leq q(x) \},
	\end{equation*}
	and
	\begin{equation*}
		B = \{(x,t)\in X \times \R_{\geq 0}\ |\ t \geq q(x) \}.
	\end{equation*}
	Then $X=A\cup B$ is a coarsely excisive decomposition.
\end{proposition}


\section{Coarse Homotopy}

With this abstract notion of a coarse space in place, we can now define a notion of coarse homotopy theory. The following definition is adapted from Definition 2.1 of \cite{mitchener2020coarse}, the difference is only a small technical modification.

\begin{definition}
	Let $X$ be a coarse space and $q:X \to \R_{\geq 0}$ be a coarse map. We define the \textbf{$\boldsymbol{q}$-cylinder on X} by the set
	\begin{equation*}
		I_qX \coloneqq \{(x,t) \in X \times \R_{\geq 0}\ |\ t \leq q(x)\}
	\end{equation*}
	with the coarse subspace structure inherited from the product structure on $X \times \R_{\geq 0}$.
	
	Note, the inclusion maps $i_0, i_q : X \to I_qX$ defined by $i_0(x) \coloneqq (x,0)$ and $i_q(x) \coloneqq (x, q(x))$ are both coarse. We then say that a pair of parallel coarse maps $f,g:X \to Y$ between coarse spaces $X$ and $Y$ are \textbf{coarse homotopic} if there exists coarse maps $q:X \to \R_{\geq 0}$ and $H_q:I_qX \to Y$ such that the diagram
	\begin{center}
		\begin{tikzcd}
			& Y & \\
			X \arrow[r,"i_0" swap] \arrow[ur, "f"]& I_qX \arrow[u,"H_q" description] & X \arrow[l, "i_q"] \arrow[ul, "g" swap]
		\end{tikzcd}
	\end{center}
	commutes. 
	
	Further, we call a coarse map $f:X \to Y$ a \textbf{coarse homotopy equivalence} if there exists a further coarse map $f':Y \to X$ such that $gf$ is coarse homotopic to $1_X$ and $fg$ is coarse homotopic to $1_Y$. In which case, we say that $X$ and $Y$ are coarse homotopy equivalent.
\end{definition}

\begin{theorem}{(Theorem 2.4 in \cite{mitchener2020coarse})}\label{crs_hom_equiv_rel}
	The relation of coarse homotopy is an equivalence relation, denoted $\overset{\text{Crs}}{\sim}$.
\end{theorem}

\begin{example}[Example 2.3, \cite{mitchener2020coarse}]\label{close_is_crs_hom}
	Let $X,Y$ be coarse spaces, and let $f,g:X \to Y$ be a pair of close coarse maps, and such that there exists some coarse map $q:X \to \R_{\geq 0}$. The map $H:I_{q+1}X \to Y$ defined by
	\begin{align*}
		H_q(x,s)\coloneqq \begin{cases}
			f(x) & s<1,\\
			g(x) & s\geq 1,
		\end{cases}
	\end{align*}
	is a coarse homotopy from $f$ to $g$. Further, if $f$ is a coarse equivalence, with some coarse inverse $f':Y \to X$. Then $f'f$ is close (and hence coarse homotopic) to $1_X$ and $ff'$ is close (and hence coarse homotopic) to $1_Y$, and so coarse equivalences are coarse homotopy equivalences.
\end{example}

\begin{proposition}\label{comp_crs_hom}
	Let $W,X,Y$, and $Z$ be a coarse spaces, $f:W \to X$, $g,h : X \to Y$, and $k:Y \to Z$ be coarse maps, and $H_q:I_q X \to Y$ be a coarse homotopy from $g$ to $h$ for some coarse map $q:X \to \R_{\geq 0}$. Then there is a coarse homotopy from $gf$ to $hf$, and from $kg$ to $kh$, respectively.
\end{proposition}

\begin{proof}
	The first homotopy follows from the fact that the diagram
	\begin{center}
		\begin{tikzcd}
			& Y & \\
			X \arrow[r,"i_0" swap] \arrow[ur, "g"]& I_qX \arrow[u,"H_q" description] & X \arrow[l, "i_q"] \arrow[ul, "h" swap]\\
			W \arrow[u, "f"] \arrow[r,"i_0" swap] & I_{fq}W \arrow[u, "f\times 1_{\R}" description] & W \arrow[u, "f" swap] \arrow[l, "i_{fq}"]
		\end{tikzcd}
	\end{center}
	commutes, and the second homotopy follows from the fact that the diagram
	\begin{center}
		\begin{tikzcd}
			& Z & \\
			& Y \arrow[u, "k" description]& \\
			X \arrow[r,"i_0" swap] \arrow[ur, "g"] \arrow[uur, bend left, "kg"] & I_qX \arrow[u,"H_q" description] & X \arrow[l, "i_q"] \arrow[ul, "h" swap] \arrow[uul, bend right, "kh" swap]
		\end{tikzcd}
	\end{center}
	also commutes.
\end{proof}

\begin{proposition}
	The relation of coarse homotopy equivalence is symmetric and transitive on the class of coarse spaces. Further, it is reflexive, and hence an equivalence relation, on the subclass of coarse spaces that have at least one coarse map with codomain $\R_{\geq 0}$.
\end{proposition}

\begin{proof}
	For reflexivity, the identity is a coarse equivalence and hence a coarse homotopy equivalence by Example \ref{close_is_crs_hom}. Symmetry is similarly straightforward to show straight from the definition. Finally, to see coarse equivalence is transitive let $X,Y$, and $Z$ be coarse spaces and $f:X \to Y$ and $g:Y \to Z$ be coarse homotopy equivalences with coarse homotopy inverses $f':Y \to X$ and $g':Z \to Y$ respectively. Then, $gf$ is a coarse homotopy equivalence with coarse homotopy inverse $f'g'$, since $f'g'gf\overset{\text{Crs}}{\sim} f'1_Yf = f'f \overset{\text{Crs}}{\sim} 1_X$, using Proposition \ref{comp_crs_hom}. It can similarly be shown that $gff'g' \overset{\text{Crs}}{\sim} 1_Z$.
\end{proof}

\begin{remark}
	Note, not all coarse spaces have a coarse map into $\R_{\geq 0}$. This is an unfortunate aspect of this definition of coarse homotopy as it is possible for a pair of spaces to be equivalent but not coarse homotopy equivalent. Fortunately, for all the cases we are interested in this is not a problem. In particular, for any metric space $X$ and fixed base-point $x_0$, the map $X\to \R_{\geq 0}$ defined by $x\mapsto d_X(x_0,x)$ is coarse.
\end{remark}

We will again abuse terminology and notation and define a coarse homotopy between metric coarse maps (Definition \ref{metric_coarse_map_dfn}), to mean a coarse homotopy between induced coarse maps (Example \ref{metric_crs_=_crs_crs}) with respect to the metric coarse structures (Example \ref{coarse_space_basic_examples} (1)), and similarly for coarse homotopy equivalences between metric spaces.

\begin{proposition}\label{Rn+_hom_eq_to_R+}
	Let $n \geq 1$, then there is a coarse homotopy equivalence from $\R_{\geq 0}^n$ to $\R_{\geq 0}$.
\end{proposition}

\begin{proof}
	We will show that for all $n\geq 2$, there is a coarse homotopy equivalence from $\R_{\geq 0}^n$ to $\R_{\geq 0}^{n-1}$, and then by a basic induction argument the result follows. The natural thought is to reduce the dimension of $\R_{\geq 0}^n$ by way of a usual projection, however the canonical projection maps are not proper. Instead, notice for each $n\geq 1$ there is an isometry (and hence a coarse equivalence and coarse homotopy equivalence)
	\begin{equation*}
		X_n \coloneq \{(x_1,...,x_n)\in \R^n\ |\ x_1 \geq 0, \text{ and } |x_i|\leq x_1 \text{ for } i \in\{2,...,n\}\} \cong \R^n_{\geq 0},
	\end{equation*}
	via a $-\frac{\pi}{4}$ rotation in each dimension. Then, we will show there is a coarse homotopy equivalence from $X_n$ to $X_{n-1}$ $\big($and hence from $\R_{\geq 0}^n$ to $\R_{\geq 0}^{n-1}\big)$. Define a pair of maps $\text{inc}:X_{n-1} \to X_n$ and $\pi:X_n \to X_{n-1}$, by 
	\begin{equation*}
		\text{inc}(x_1,...,x_{n-1}) \coloneq (x_1,...,x_{n-1},0)
	\end{equation*}
	and 
	\begin{equation*}
		\pi(x_1,...,x_{n-1},x_n) \coloneqq (x_1,...,x_{n-1})
	\end{equation*}
	respectively.
	
	The inclusion inc is easily seen to be coarse. Further, $\pi$ is easily seen to be controlled. To see it is proper let $B \in \B_X$. Let $x=\text{sup}_{(x_1,...,x_{n-1})\in \B}(x_1)$. Then
	\begin{equation*}
		B \subseteq \{(x_1,...,x_n)\in \R^n\ |\ x_1 \leq x, \text{ and } |x_i|\leq x_1 \text{ for } i \in\{2,...,n\}\} \in \B_{X_{n-1}},
	\end{equation*}
	then
	\begin{align*}
		\pi^{-1}[B] & \subseteq \pi^{-1}[\{(x_1,...,x_{n-1})\in \R^n\ |\ x_1 \leq x, \text{ and } |x_i|\leq x_1 \text{ for } i \in\{2,...,n-1\}\}]\\
		& = \{(x_1,...,x_{n})\in \R^n\ |\ x_1 \leq x, \text{ and } |x_i|\leq x_1 \text{ for } i \in\{2,...,n\}\} \in \B_{X_n},
	\end{align*}
	and $\pi$ is proper. A very similar argument can be used to show the projection $\pi_{x_1}:X_n \to \R_{\geq 0}$ where $\pi_{x_1}(x_1,...,x_n)\coloneq x_1$ is coarse.
	
	Then, it is quick to see $\pi \circ \text{inc} = 1_{X_{n-1}}$. Then, the map $H:I_{\pi_{x_1}}X_n \to X_n$ defined as follows gives a coarse homotopy from $\text{inc}\circ \pi$ to $1_{X_n}$.
	\begin{equation*}
		H\big((x_1,...,x_n),s\big) \coloneq \begin{cases}
			(x_1,...,x_{n-1},x_n -s) & x_n \geq s,\\
			(x_1,...,x_{n-1},x_n +s) & x_n \leq - s,\\
			(x_1,...,x_{n-1},0) & |x_n| \leq s.
		\end{cases}
	\end{equation*}
	Then $H$ is coarse by Proposition \ref{crs_excisive_for_R}, and we are done.
\end{proof}


\section{Coarse Path Components}

With a notion of coarse homotopy, an immediate next step is to consider coarse homotopy invariants. This is exactly where the authors of \cite{mitchener2020coarse} take this be defining a coarse notion of the usual homotopy groups. Here, we are only concerned with coarse path components.

Before we can define these, however, we need a notion of a coarse path. The following construction which is designed to promote small-scale structures to large-scale ones, including the usual topological $n$-dimensional unit cube, will aid in this endeavour.

\begin{definition}[Definition 3.1, \cite{mitchener2020coarse}]\label{cone_definition}
	Let $X$ be a bounded subset of $\mathbb{R}^n$ with the usual $l^2$ metric, $R>0$. We define the \textbf{metric cone of $X$ truncated at $R$} by
	\begin{equation*}
		c_R(X) \coloneqq \{(hx,h) \in \mathbb{R}^{n+1}\ |\ h\geq R,\ x \in X \},
	\end{equation*}
	equipped with the induced subspace metric. In the case where $R=0$ we will simply denote $c_0(X)$ by $c(X)$. The inclusion $c_R(X) \hookrightarrow c(X)$ is a coarse equivalence. Further, we will denote the canonical continuous and controlled (but not proper) projection by
	\begin{align*}
		\text{proj}_R : c_R(X) & \to X,\\
		(hx,h) & \mapsto x,
	\end{align*}
	for each $R>0$, and the canonical continuous, controlled, and proper inclusion by
	\begin{align*}
		\text{inc}_h : X & \to c_R(X), \\
		x & \mapsto (hx,h)
	\end{align*}
	for each $R\geq 0$ and $h \geq R$. Finally, we will define the $h$\textbf{-layer} of $c_R(X)$ to be the subset
	\begin{equation*}
		\big\{(h'x,h') \in c_R(X)\ |\ h=h' \big\}
	\end{equation*}
	for some fixed $h\geq R$.
\end{definition}

\begin{example}
	For each $n \in \Z_{\geq 1}$ there is a coarse isomorphism $c(S^n)\cong \R^{n+1}$. Further, there is a coarse isomorphism $c_1([0,1]) \cong I_{1_{\R_{\geq 0}}+1}\R_{\geq 0}$, defined by simply by an affine shift of the cylinder.
\end{example}	

\begin{definition}\label{cone_ray}
	Let $X$ be a bounded subset of $\R^n$ for some $n\geq 1$ and let $x\in X$. Then, for any $R\geq 0$, we will denote the canonical ray $\R_{\geq 0} \to c_R(X)$ defined by $h \mapsto (hx,h)$ for $h\geq R$ and $h \mapsto (Rx,R)$ for $h\leq R$, by $\alpha_x$.
\end{definition}

Notice $\alpha_x$ is $1$-Lipschitz, and an isometry when restricted to $\alpha_x|_{h\geq R}$.

\begin{proposition}\label{c(X)_is_proper}
	Let $X$ be a compact subset of $\R^n_{\geq0}$ for some $n\geq 0$. The $c(X)$ is proper. 
\end{proposition}

\begin{proof}
	Since $X$ is compact, it is closed, and so $c(X)$ is a closed subset of $\R^n_{\geq0}$, and thus inherits properness from the euclidean metric on $\R^n_{\geq0}$.
\end{proof}

Now we can define a coarse notion of a path.

\begin{definition}\label{pathdfn}
	Let $X$ be a coarse space. Let $\alpha,\alpha':\R_{\geq 0} \to X$ be coarse rays in $X$. A \textbf{coarse $1$-path} $\varphi$ from $\alpha$ to $\alpha'$ is a coarse map $\varphi : c_1([0,1]) \to X$, such that $\varphi|_{c_1(\{0\})} = \alpha\circ \pi_h$ and $\varphi|_{c_1(\{1\})} = \alpha'\circ \pi_h$, denoted $\varphi : \alpha \rightsquigarrow \alpha'$,
\end{definition}

We will often refer to a coarse $1$-path $\varphi : c_1([0,1]) \to X$ from $\alpha$ to $\alpha'$ by just $\varphi$, terming $\alpha$ and $\alpha'$ the \textbf{endrays} of $\varphi$.

\begin{proposition}\label{nplaneabovecrs}
	Let $q: \R_{\geq 0} \to \R_{\geq 0}$ be a coarse map. Then there exists some $A_1,B_1 \geq 1$ such that $q(h)\leq A_1h+B_1$ for all $h \in \R_{\geq 0}$.
\end{proposition}

\begin{proof}
	By Proposition \ref{aL_implies_crd}, $q$ is asymptotically Lipschitz, and so there exits some $A_0, B_0 \geq 1$ such that
	\begin{align*}
		d_{\R_{\geq 0}}(q(h),q(h')) \leq A_0\cdot d_{\R_{\geq 0}}(h,h')+B_0,
	\end{align*}
	for all $h,h' \in \R_{\geq 0}$. Then
	\begin{align*}
		q(h) & = d_{\R_{\geq 0}}(q(h),0)\\ 
		& \leq d_{\R_{\geq 0}}(q(h),q(0)) + d_{\R_{\geq 0}}(q(0),0)\\
		& \leq A_0\cdot d_{\R_{\geq 0}}(h,0)+B_0 + d_{\R_{\geq 0}}(q(0),0))\\
		& \leq A_0h +B_0 + d_{\R_{\geq 0}}(q(0,...,0),(0,...,0)),
	\end{align*}
	and hence simply define $A_1\coloneqq A_0$ and $B_1 \coloneqq B_0 + d_{\R_{\geq 0}}(q(0,...,0),(0,...,0))$.
\end{proof}

\begin{corollary}\label{homiffpath}
		Let $q: \R_{\geq 0} \to \R_{\geq 0}$ be a coarse map. Then there is a canonical coarse embedding $I_q\R_{\geq 0} \subseteq I_{\pi_h + 1}\R_{\geq 0}$. Further, if $X$ be a coarse space with coarse rays $\alpha, \alpha' : \R_{\geq 0} \to X$. Then there exists a coarse $1$-path $\varphi : \alpha \rightsquigarrow \alpha'$ if and only if the maps $\alpha, \alpha'$ are coarsely homotopic.
\end{corollary}

\begin{proof}
	Let $A,B \geq 1$ such that $q(h) \leq Ah + B$ for all $h$. Let $A' = \text{max}(A,B)$. Then the map $I_q\R_{\geq 0} \to I_{\pi_h + 1}\R_{\geq 0}$ defined $(h,s) \mapsto \big(h,\frac{s}{A}\big)$ gives the desired embedding. 

	Now suppose there is a coarse homotopy $H_q : I_q\R_{\geq 0} \to X$ between $\alpha$ and $\alpha'$. Then by the coarse isomorphism $I_{1_{\R_{\geq 0}} + 1}\R_{\geq 0} \cong c_1([0,1])$ we get a canonical embedding $I_q\R_{\geq 0} \hookrightarrow c_1([0,1])$. Then, the map $\varphi^A_q : c_1([0,1]) \to X$ defined  by
	\begin{align*}
		\varphi_q^A(ht,h) \coloneqq \begin{cases} 
			H_q(h,Aht) & 0 \leq ht \leq \frac{q(h)}{A},\\
			\alpha'(h) & \frac{q(h)}{A} \leq ht \leq h.
		\end{cases}
	\end{align*}
	is a coarse $1$-path $\varphi^A_q: \alpha \rightsquigarrow \alpha'$. For the converse, assume there is a coarse $1$-path $\varphi: \alpha \to \alpha'$. Then the composite
	\begin{equation*}
		I_{\pi_h+1}\R_{\geq 0} \xrightarrow{\cong} c_1([0,1]) \xrightarrow{\varphi} X
	\end{equation*}
	gives the desired coarse homotopy from $\alpha$ to $\alpha'$.
\end{proof}

\begin{proposition}\label{dimatatime}
	Let $X$ be a coarse space and $f:c_1([0,1])\to X$ be some map. Then $f$ is controlled if and only if for all $R>0$, $f\times f$ preserves the following family of controlled sets in $c_1([0,1])$,
	\begin{equation*}
		U_R \coloneqq \big\{\big((ht,h),(ht',h)\big)\ |\ d_{c_1([0,1])}\big((ht,,h),(ht',h)\big)<R \big\},
	\end{equation*}
	and
	\begin{equation*}
		V_R \coloneqq \left\{\big((ht,h),(ht,h')\big)\ |\  d_{c_1([0,1])}\big((ht,h),(ht,h')\big)<R \right\}.
	\end{equation*}
\end{proposition}

\begin{proof}
	The forward implication is trivial, and the backwards implication is just an application of the triangle inequality.
\end{proof}

With the definition of coarse $1$-paths, there is natural question to ask related to coarse path components. The following definition and examples can be found in \cite{mohamad2013coarse}.

\begin{definition}\label{pi01}
	Let $X$ be a coarse space. The set of \textbf{coarse path components} of $X$, denoted $\pi^{\text{Crs}}_0(X)$, is the set of coarse homotopy classes of maps $\R_{\geq 0} \to X$.
\end{definition}

\begin{example}
	\begin{enumerate}
		\item The set $\pi_0^{\text{Crs}}(\R_{\geq 0})$ has one element represented by $1_{\R_{\geq 0}}$.
		\item The set  $\pi_0^{\text{Crs}}(\mathbb{R})$ has two elements represented by the maps $x \mapsto x$ and $x \mapsto -x$ for all $x \in \R_{\geq 0}$ respectively.
		\item Let $B$ be a bounded space. Then $\pi_0^{\text{Crs}}(B) = \emptyset$ since there are no coarse maps $\R_{\geq 0} \to B$.
	\end{enumerate}
\end{example}

\begin{definition}
	Consider the geometric realization $|X|$ of some finite simplicial complex $X$, viewed as a subset of $\R^n$ for some $n\in \Z_{\geq 1}$ via barycentric co-ordinates (see section 2.1 of \cite{hatcheralgtop} for details). We will call a well-defined embedding of a realised finite simplicial complex into $\R^m$, for some potentially different $m\in \Z_{\geq 1}$, \textbf{piecewise linear} if it is linear with the respect to the barycentric co-ordinates when restricted to each simplex (see \cite{brehm1992linear} for details on piecewise linear embeddings). 
	
	By a \textbf{geometric finite simplicial complex in $\R^n$} we mean the image of a finite simplicial complex under its geometric realization and its piecewise linear embedding equipped with the subspace metric.
\end{definition}

\begin{theorem}[Theorem 5.6, \cite{mitchener2020coarse}]\label{crs_hom_grps_of_cone}
	Let $X$ be a geometric finite simplicial complex. Then there is a bijection $\pi_0(X) \cong \pi_0^{\text{Crs}}(c_1(X))$.
\end{theorem}

Like in the topological setting, there is another natural way we may define the set of coarse path components of a coarse space using the paths we defined in Definition \ref{pathdfn}. We first need the following.

\begin{definition}\label{pathconcat}
	Let $X$ be a coarse space and $\alpha : \R_{\geq 0} \to X$ be a coarse ray in $X$. We define the \textbf{constant coarse $1$-path at} $\alpha$, denoted $K_{\alpha} : \alpha \rightsquigarrow \alpha$, by $K_{\alpha}(ht,h) \coloneqq \alpha(h)$ for all $(ht,h) \in c_1([0,1])$.
	
	Further let $\varphi : \alpha \rightsquigarrow \alpha'$ be a coarse $1$-path from $\alpha$ to another coarse ray $\alpha'$ in X. Then we define its \textbf{inverse coarse $1$-path} $\varphi^{-1} : \alpha' \rightsquigarrow \alpha$ by $\varphi^{-1}(ht,h)\coloneqq \varphi(h-ht,h)$ for all $(ht,h) \in c_1([0,1])$.
	
	Finally, let $\alpha''$ be another coarse ray in $X$ such that there exists a path $\varphi' : \alpha' \rightsquigarrow \alpha''$. Then we define their \textbf{coarse $1$-path composition}, denoted $\varphi*_1\varphi' : \alpha  \rightsquigarrow \alpha''$ by
	\begin{align*}
		\varphi*_1\varphi'(ht,h) \coloneqq \begin{cases}
			\varphi\big(2ht,h\big) & 0\leq ht \leq \frac{1}{2}h,\\
			\varphi'\big(2ht - h,h\big) & \frac{1}{2}h \leq ht \leq h,
		\end{cases}
	\end{align*}
	for all $(ht,h) \in c_1([0,1])$.
\end{definition}

All maps defined above are indeed coarse; the constant path and inverse path are easily seen to be coarse, and the composite path is coarse by Proposition \ref{crs_excisive_for_R}.

\begin{proposition}
	Let $X$ be a coarse space. Then there is an equivalence relation defined on the set of coarse rays in $X$ where for each pair of coarse rays $\alpha,\alpha' : \R_{\geq 0} \to X$, we have $\alpha \sim \alpha'$ if and only if there exists a coarse $1$-path $\varphi : \alpha \rightsquigarrow \alpha'$.
	
	Further, the set of equivalence classes is equivalent to $\pi_0^{\text{Crs}}(X)$.
\end{proposition}

\begin{proof}
	The constant path gives reflexivity, the inverse path symmetry, and the composite path transitivity. The second part is true by Proposition \ref{homiffpath}.
\end{proof}

We will denote a coarse path component class by $[\alpha]_0^{\text{Crs}}$ for some representative coarse ray $\alpha$.

\begin{proposition}
	Let $f:X \to Y$ be a coarse map of coarse spaces. Then there is a well-defined induced function $f_*:\pi_0^{\text{Crs}}(X) \to \pi_0^{\text{Crs}}(Y)$ defined by $[\alpha]_0^{\text{Crs}} \mapsto [f\circ \alpha]_0^{\text{Crs}}$. 
	
	Further, the following data defines a functor $\pi_0^{\text{Crs}}(-) : \textbf{Crs} \to \textbf{Set}$,
	\begin{enumerate}
		\item $\pi_0^{\text{Crs}}(-): X \mapsto \pi_0^{\text{Crs}}(X)$, for each object $X\in \textbf{Crs}$, and
		\item $\pi_0^{\text{Crs}}(-): (f: X \to Y) \mapsto \big(f_*: \pi_0^{\text{Crs}}(X) \to  \pi_0^{\text{Crs}}(Y)\big)$, for each morphism $f$ in \textbf{Crs}.
	\end{enumerate}
\end{proposition}

\begin{proof}
	To see $f_*$ is well-defined it is enough to notice that if $\alpha$ and $\alpha'$ represent the same class, then there is a coarse $1$-path $\varphi: \alpha \myrightsquigarrow \alpha'$, and so there is a coarse $1$-path $f\varphi : f\alpha \myrightsquigarrow f\alpha'$.
	
	It is then quick to show the existence of the functor $\pi_0^{\text{Crs}}(-)$.
\end{proof}

\begin{lemma}\label{every_crs_ray_close_to_plip_ray}
	Let $X$ be a proper geodesic space. Then every coarse ray $\alpha:\R_{\geq 0} \to X$ is close to some proper Lipschitz ray, and as such, each element of $\pi_0^{\text{Crs}}(X)$ may be represented by a proper Lipschitz ray.
\end{lemma}

\begin{proof}
	By Proposition \ref{aL_implies_crd} we can assume $\alpha : \R_{\geq 0} \to X$ is asymptotically Lipschitz. Consider the map $\alpha^*: \R_{\geq 0} \to X$ defined in Definition \ref{continuousray}. The claim is that $\alpha^*$ is Lipschitz and is close to $\alpha$.
	
	First, we see that $\alpha^*|_{\Z_{\geq 0}}=\alpha|_{\Z_{\geq 0}}$ is Lipschitz by Proposition \ref{discretealislip}. Further, for each $h \in Z_{\geq 0}$ the map $\alpha^*|_{[h,h+1]}$ is equal to the composite
	$$[h,h+1] \xrightarrow{-h} [0,1] \xrightarrow{\times a_{\alpha(h),\alpha(h+1)}} [0,a_{\alpha(h),\alpha(h+1)}] \xrightarrow{u_{\alpha(h),\alpha(h+1)}} X.$$
	Each component above is Lipschitz, and so $\alpha^*|_{[h,h+1]}$ is Lipschitz (for each $h$, $\alpha^*|_{[h,h+1]}$ is $(A+B)$-Lipschitz).
	
	To see $\alpha^*$ is Lipschitz, let $h,h' \in \R_{\geq 0}$. Without loss of generality we can assume $h' > h$ (the $h=h'$ case is trivial). Then
	\begin{align*}
		d_X(\alpha^*(h),\alpha^*(h')) & = d_X(\alpha^*|_{[\ceil{h},\ceil{h}+1]}(h),\alpha^*|_{[\ceil{h},\ceil{h}+1]}(\ceil{h})) \\
		& \hspace*{2.5cm} + d_X(\alpha^*|_{\Z}(\ceil{h}),\alpha^*|_{\Z}(\floor{h'}))\\
		& \hspace*{3.5cm} + d_X(\alpha^*|_{[\floor{h'},\floor{h'}+1]}(\floor{h'}),\alpha^*|_{[\floor{h'},\floor{h'}+1]}(\floor{h'}))\\
		& \leq A|h-\lceil h \rceil| + A'|\lceil h \rceil - \lfloor h' \rfloor| + A''|\lfloor h' \rfloor - h'|\\
		& \leq \text{max}\{A,A',A''\}\left(|h-\lceil h \rceil| + |\lceil h \rceil - \lfloor h' \rfloor| + |\lfloor h' \rfloor - h'|\right)\\
		& = \text{max}\{A,A',A''\}\left(|h-h'| \right),
	\end{align*}
	for some $A,A',A'' \geq 0$, and $\alpha^*$ is $(\text{max}\{A,A',A''\})$-Lipschitz. 
	
	To see $\alpha^*$ is proper let $B\in \B_X$. Then $B \subseteq CB(\alpha(0),R)$ for some $R>0$. Suppose for the sake of contradiction that $\alpha^*$ is not proper. Then there exists a strictly increasing unbounded sequence $(h_i)_{i \in \Z_{\geq 0}}$ of real numbers such that $\{\alpha^*(h_i)\ |\ i \in \Z_{\geq }\} \subseteq CB(\alpha(0),R)$. But, by the construction of $\alpha^*$ by geodesics, and the fact that $\alpha^*$ is $A$-Lipschitz for some $A\geq 0$, we have
	\begin{equation*}
		\{\alpha(\ceil{h_i})\ |\ i\in \Z_{\geq 0}\} = \{\alpha^*(\ceil{h_i})\ |\ i\in \Z_{\geq 0}\} \subseteq CB(\alpha(0),R+A) \in \B_X,
	\end{equation*}
	and hence we reach a contradiction since $\alpha$ is proper.
	
	Finally, we will show $\alpha^*$ is close to $\alpha$. Let $h \in \R_{\geq 0}\backslash \Z_{\geq 0}$ (the $h \in \Z_{\geq 0}$ is trivial). Then
	\begin{align*}
		d_X(\alpha^*(h),\alpha(h)) & \leq d_X\left(\alpha^*|_{[\floor{h},\ceil{h}]}(h),\alpha^*|_{[\floor{h},\ceil{h}]}(\floor{h})\right) + d_X(\alpha(\floor{h}),\alpha(h))\\
		& \leq A|h-\floor{h}| + A' |h-\floor{h}|\\
		& \leq A+A',
	\end{align*}
	since $|h-\floor{h}|<1$, for some $A,A' >0$. Hence $\alpha^*$ and $\alpha$ are close and every class in $\Pi_0^{\text{Crs}}(X)$ can be represented by a proper Lipschitz ray.
\end{proof}

\begin{corollary}\label{any_root}
	Let $X$ be a proper geodesic space and fix some $x_0 \in X$. Then each coarse ray $\alpha$ in $X$ is close to some proper Lipschitz ray $\alpha_{x_0}^{*}:\R_{\geq 0} \to X$ rooted at $x_0$, that is, such that $\alpha_{x_0}^{*}(0) = x_0$. As such, each element of $\pi_0^{\text{Crs}}(X)$ may be represented by an $x_0$-rooted proper Lipschitz ray.
\end{corollary}

\begin{proof}
	By Lemma \ref{every_crs_ray_close_to_plip_ray} above $\alpha$ is close to some proper Lipschitz ray $\alpha^*$. We then may define $\alpha_{x_0}^{*}$ by
	\begin{align*}
		\alpha_{x_0}^{*}(h) \coloneqq \begin{cases}
			u_{x_0,\alpha(1)}(a_{x_0,\alpha(1)}h) & 0 \leq h \leq 1,\\
			\alpha^*(h) & 1 \leq h.
		\end{cases}
	\end{align*}
	In particular if $\alpha^*$ is proper $A$-Lipschitz, $\alpha_{x_0}^{*}$ is proper $\text{max}\{A,a_{x_0,\alpha(1)}\}$-Lipschitz. It is quick to show $\alpha_{x_0}^{*}$ and $\alpha^{*}$ are close.
\end{proof}


\section{Coarse Homotopy and Ends of Proper Geodesic Spaces}\label{section_7}

We can now begin to consider the main topic of this paper, that is the interactions between ends and coarse homotopy. While proper geodesic spaces make up the main class of spaces of interest, first we will consider ends in the context of metric cones. Recall the definition of the ray $\alpha_x$ defined on the cone of a point (Definition \ref{cone_ray}).

\begin{proposition}\label{top_pi0_is_ends_of_cone}
	Let $X$ be a geometric finite simplicial complex. Then, there is a bijection $\mathcal{E}\text{nds}(c(X))\cong \pi_0(X)$.
\end{proposition}

\begin{proof}
	Define a map $F : \pi_0(X) \to \ends(c(X))$ by $[x]_0 \mapsto \text{end}(\alpha_x)$. We claim that this is a bijection. To see $F$ is well-defined let $[x]_0=[x']_0$ for $x,x'\in X$, $u:[0,1] \to X$ be some continuous path from $x$ to $x'$, and $B$ be some bounded set in $c(X)$. Note, there exists some $R>0$ such that $B\subseteq c(X)\backslash c_R(X)$. Consider $c(X)\backslash(c(X)\backslash c_R(X)) = c_R(X)$. Then, for any $h\geq R$, the composition continuous path $\text{inc}_h \circ u$ defines a continuous path from $\alpha_x(h)$ to $\alpha_{x'}(h)$, and so we have that $\text{end}(\alpha_x(h)) = \text{end}(\alpha_{x'}(h))$ and $F$ is well-defined.
	
	Next, let $x,x'\in X$ such that $\text{end}(\alpha_x)=\text{end}(\alpha_{x'}(h))$. Then for each $R>0$, there is a some $h\geq R$ and a path $u:[0,1]\to c_R(X)$ from $\alpha_x(h)=(hx,h)$ to $\alpha_{x'}(h)=(hx',h)$, and so the composite $\text{proj}_R \circ u$ defines a path from $x$ to $x'$, so $[x]_0=[x']_0$ and $F$ is injective.
	
	Finally, let $\alpha :\R_{\geq 0} \to c(X)$ be a topologically proper (and hence metrically proper by Proposition \ref{metric_proper_is_top_proper} and Proposition \ref{c(X)_is_proper}) continuous ray. By properness, there exits $S>0$ such that $\text{im}\big(\alpha|_{[S,\infty)}\big) \cap \boldsymbol{0} = \emptyset$. Choose some $x$ such that $\alpha(S')=(hx,h)$ for some $S'>S$, and some $h \geq 0$. We claim $\text{end}(\alpha)=\text{end}(\alpha_x)$. Again, let $B$ be some bounded set contained in $c(X)\backslash c_R(X)$ for some $R>0$ and consider $c(X)\backslash(c(X)\backslash c_R(X))=c_R(X)$. Using the properness of $\alpha$ there exists $S''>0$ such that $\text{im}\big(\alpha|_{[S,\infty)}\big) \cap c(X)\backslash c_R(X) = \emptyset$.  Let $T>\text{max}(S'',R)$. We can define a continuous path $u:[0,1]\to c_h(X)$ from $\alpha_x(T)$ to $\alpha(T)$ in $c_h(X)$ defined by 
	\begin{equation*}
		u(t) \coloneqq \begin{cases}
			\big(((1-2t)T+2th)x,(1-2t)T+2th\big) & 0 \leq t \leq \frac{1}{2},\\
			\alpha\big((2-2t)s+(2t-1)T\big) & \frac{1}{2} \leq t \leq 1,
		\end{cases}
	\end{equation*}
	for all $t$. Then the composite $\text{inc}_T\circ \text{proj}_R\circ u$ defines a continuous path from $\alpha_x(T)$ to $\alpha(T)$ in $c_R(X)$, so $\text{end}(\alpha)=\text{end}(\alpha_x)$, $F$ is surjective, and hence a bijection.
\end{proof}

\begin{corollary}
	Let $X$ be a geometric finite simplicial complex. Then there exists a bijection $\mathcal{E}\text{nds}(c(X))\cong \pi^{\text{Crs}}_0(c(X))$.
\end{corollary}

\begin{proof}
	The composition of $\mathcal{E}\text{nds}(c(X))\cong \pi_0(X) \cong \pi^{\text{Crs}}_0(c(X))$, with bijections given in Proposition \ref{top_pi0_is_ends_of_cone} and Theorem \ref{crs_hom_grps_of_cone} gives the desired result.
\end{proof}

This is a nice example that shows that there are at lease some cases where the notions of coarse path components and ends align. This then of course extends to any space coarse homotopy equivalent to a metric cone of a geometric finite simplicial complex. We would like to, however, consider spaces beyond this, in particular proper geodesic metric spaces.

\begin{proposition}\label{geospace_pi0=end}
	Let $X$ be a geodesic metric space and $\alpha,\alpha':\R_{\geq 0} \to X$ be metrically proper and topologically proper Lipschitz maps such that $[\alpha]_0^{\text{Crs}}=[\alpha']_0^{\text{Crs}}$. Then we have $\text{end}(\alpha)=\text{end}(\alpha')$.
\end{proposition}

\begin{proof}
	Without loss of generality we can assume $\alpha(0) = \alpha'(0)=x_0$ for some $x_0\in X$. Let $K\subseteq X$ be a compact subset. Then $K\subseteq CB(x_0,R)\subseteq X$ for some $R>0$. By Proposition \ref{homiffpath} and Proposition \ref{aL_implies_crd} there exists a proper and $(A,B)$-asymptotically Lipschitz coarse $1$-path $\varphi:\alpha \myrightsquigarrow \alpha'$.
	
	For each $h\in \R_{\geq 0}$ we can define a topological path $u_h:[0,1]\to X$ from $\alpha(h)$ to $\alpha'(h)$ as follows.
	\begin{align*}
		u_h(t)\coloneqq \begin{cases}
			u_{\varphi(0,h),\varphi(1,h)}(a_{\varphi(0,h),\varphi(1,h)}ht) & 0 \leq ht \leq 1,\\
			\hspace*{2cm}  \vdots & \hspace*{1cm} \vdots \\
			u_{\varphi(j,h),\varphi(j+1,h)}(a_{\varphi(j,h),\varphi(j+1,h)}(ht-j)) & j \leq ht \leq j+1,\\
			\hspace*{2cm} \vdots &\hspace*{1cm} \vdots \\
			u_{\varphi(h-1,h),\varphi(h,h)}(a_{\varphi(h-1,h),\varphi(h,h)}(ht-h+1)) & h-1 \leq ht \leq h,\\
		\end{cases}
	\end{align*}
	where each $u_{\varphi(j,h),\varphi(j+1,h)}$ is a choice of geodesic segment. We claim there exists some $h$ such that for all $h'\geq h$ we have $\text{im}(u_{h'})\cap (X\backslash CB(x_0,R))=\emptyset$, and thus $u_{h'}$ is a topological path connecting $\text{im}(\alpha)\cap (X\backslash CB(x_0,R))$ and $\text{im}(\alpha')\cap (X\backslash CB(x_0,R))$.
	
	For the sake of contradiction, suppose there exists no such $h$. Then, for each $h\in \R_{\geq 0}$ there exists $x_h \in CB(x_0,R)\subseteq X$ such that $x_h=u_h(t)$ for some $t\in [0,1]$. Since $\varphi$ is proper and  $(A,B)$-asymptotically Lipschitz, for each $h\in\R_{\geq 0}$ and $j\in \{0,1,...,h-1\}$,
	\begin{equation*}
		d_X(\varphi(j,h),\varphi(j+1,h)) \leq A\cdot d_{c_1([0,1])}((j,h),(j+1,h))+B = A+B.
	\end{equation*}
	As such, for each $x_h$ there exists some $x'_h\in CB(x_h,A+B)$ such that $x_h'=\varphi(j_h,h)$ for some $j_h\in\{0,1,...,h\}$. Then $\{(j_h,h)\ |\ h\in \R_{\geq 0}\}$ is unbounded in $c_1([0,1])$ but 
	\begin{equation*}
		\varphi\big[\{(j_h,h)\ |\ h \in\ \R_{\geq 0}\}\big]\subseteq CB(x_0,R+A+B) \in \B_{X},
	\end{equation*}
	and thus $\varphi$ is not proper and we reach a contradiction.
\end{proof}

We may now state one of our main results.

\begin{theorem}\label{ends_natural_surjection}
	There exists a natural surjection $\pi^{\text{Crs}}_0(-) \twoheadrightarrow \mathcal{E}\text{nds}(-) : \textbf{PGeo$^{\text{Crs}}$} \to \textbf{Set}$.
\end{theorem}

\begin{proof}
	For each proper geodesic space $X$, define a function $\eta_X:\pi_0^{\text{Crs}}(X) \to \ends(X)$ where $\eta_X([\alpha]_0^{\text{Crs}})\coloneq \text{end}(\alpha^*)$ where $\alpha^*$ is the proper Lipschitz ray defined in Definition \ref{continuousray}. To see this is well-defined let $[\alpha]_0^{\text{Crs}}=[\alpha']_0^{\text{Crs}}$ for some coarse rays $\alpha$ and $\alpha'$ in $X$. Then 
	\begin{equation*}
		[\alpha^*]_0^{\text{Crs}} = [\alpha]_0^{\text{Crs}} = [\alpha']_0^{\text{Crs}} = [(\alpha')^*]_0^{\text{Crs}}
	\end{equation*}
	by Lemma \ref{every_crs_ray_close_to_plip_ray} and so $\text{end}(\alpha^*) = \text{end}((\alpha')^*)$ by Proposition \ref{geospace_pi0=end}. Each map $\eta_X$ is surjective since by Lemma \ref{bridsonendslemma} $(2)$ every end in $X$ has a geodesic (and hence coarse) representative.
	
	We only need $\eta_X$ to be natural in $X$. To see this, let $Y$ be another proper geodesic space and $f:X \to Y$ a coarse map. Then the diagram
	\begin{center}
		\begin{tikzcd}
			\pi_0^{\text{Crs}}(X) \arrow[r,"f_*"] \arrow[d,"\eta_X" swap]& \pi_0^{\text{Crs}}(Y) \arrow[d,"\eta_Y"]\\
			\ends(X) \arrow[r, "f_{\mathcal{E}}" swap] & \ends(Y)
		\end{tikzcd}
	\end{center}
	commutes since by Lemma \ref{every_crs_ray_close_to_plip_ray} and the fact the $f_*$ is well-defined we have
	\begin{equation*}
		[(f\circ \alpha^*)^*]_0^{\text{Crs}} = [f\circ \alpha^*]_0^{\text{Crs}} = [f\circ \alpha]_0^{\text{Crs}} = [(f\circ \alpha)^*]_0^{\text{Crs}},
	\end{equation*}
	and so $\text{end}((f\circ \alpha^*)^*) = \text{end}((f\circ \alpha)^*)$ by Proposition \ref{geospace_pi0=end}, the above diagram commutes, and the result holds.
\end{proof}

The following corollary then shows that the set of ends of a proper geodesic space is a coarse homotopy invariant.

\begin{corollary}
	The functor $\ends(-):\textbf{PGeo}^{\text{Crs}}\to \textbf{Set}$ maps coarse homotopy equivalent morphisms to equal maps, and thus coarse homotopy equivalences to bijections.
\end{corollary}

\begin{proof}
	Let $X$ and $Y$ be proper geodesic spaces and $f,g:X \to Y$ be coarse maps such that $f$ is coarse homotopic to $g$ and consider $\text{end}(\alpha) \in \ends(X)$ for some ray $\alpha$ in $X$. By Lemma \ref{bridsonendslemma} we can assume $\alpha$ is geodesic. Then, $f\circ \alpha$ and $g\circ \alpha$ are coarse rays in $Y$, and
	\begin{equation*}
		[(f\circ \alpha)^*]^{\text{Crs}}_0 = [f\circ \alpha]^{\text{Crs}}_0 = [g\circ \alpha]^{\text{Crs}}_0 = [(g\circ \alpha)^*]^{\text{Crs}}_0,
	\end{equation*}
	and so
	\begin{equation*}
		f_{\mathcal{E}}(\text{end}(\alpha)) = \text{end}((f\circ \alpha)^*) = \text{end}((g\circ \alpha)^*) = g_{\mathcal{E}}(\text{end}(\alpha)),
	\end{equation*}
	by above, and we are done.
\end{proof}

We also get the following further corollary of Theorem \ref{ends_natural_surjection}.

\begin{corollary}
	Let $\Gamma$ be a finitely-generated geometric group (through its Cayley Graph), such that $\Gamma$ can be expressed as an amalgamated free product $A*_cB$ or HNN extension $A*_C$ for finite $C$, $|A\backslash C| \geq 3$ and $|B \backslash C|\geq 2$. Then $|\pi_0^\text{Crs}(\Gamma)|$ is uncountably infinite. 
\end{corollary}

\begin{proof}
	Follows from Theorem \ref{ends_natural_surjection} and Theorem \ref{ends_of_groups} $(4)$.
\end{proof}

A natural question to ask here is whether the above map is an injection. The following shows this is not in general the case, even if we restrict as far as planar locally finite graphs.

\begin{theorem}\label{ends_neq_pi0}
	There exists a proper geodesic metric space $X$ with $|\ends(X)|< |\pi_0^{\text{Crs}}(X)|$.
\end{theorem}

\begin{proof}
	Let $X$ be the proper geodesic space defined by a pair of geodesic rays constructed by consecutive edges, connected by a vertex at on end, and $n^2$ edges connecting the $n^{\text{th}}$ vertex on each ray from the joined end, equipped with the usual path metric, depicted below.
	
	
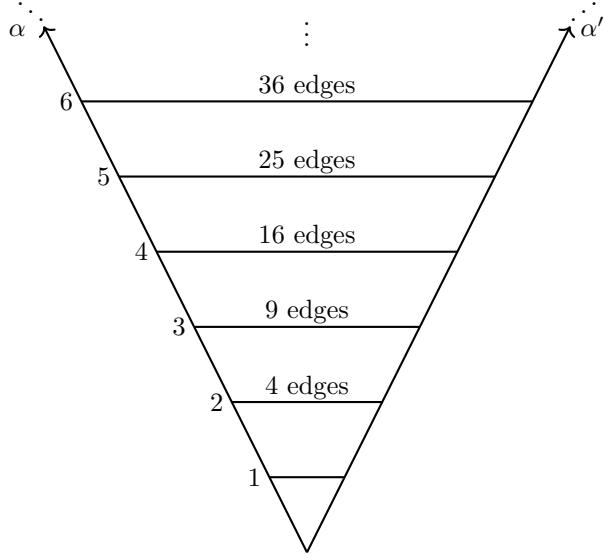
\begin{center}
		\begin{tikzpicture}
			\draw[thick,->] (0,0) -- (-3.5,7);
			\draw[thick,->] (0,0) -- (3.5,7);
			\draw[thick] (-0.5,1) -- (0.5,1);
			\draw[thick] (-1,2) -- (1,2);
			\draw[thick] (-1.5,3) -- (1.5,3);
			\draw[thick] (-2,4) -- (2,4);
			\draw[thick] (-2.5,5) -- (2.5,5);
			\draw[thick] (-3,6) -- (3,6);
			
			\node at (0,2.2) {$4$ edges};
			\node at (0,3.2) {$9$ edges};
			\node at (0,4.2) {$16$ edges};
			\node at (0,5.2) {$25$ edges};
			\node at (0,6.2) {$36$ edges};
			
			\node at (-0.7,1) {$1$};
			\node at (-1.2,2) {$2$};
			\node at (-1.7,3) {$3$};
			\node at (-2.2,4) {$4$};
			\node at (-2.7,5) {$5$};
			\node at (-3.2,6) {$6$};
			
			\node at (-3.8,7) {$\alpha\phantom{'}$};
			\node at (3.8,7) {$\alpha'$};
			
			\node at (-3.67,7.3) {$\ddots$};
			\node at (3.67,7.3) {$\iddots$};
			\node at (0, 7) {$\vdots$};
		\end{tikzpicture}
	\end{center}
	Call the collection of edges and vertices that create the horizontal step at height $n$, step$(n)$. Consider the maps $\alpha,\alpha':\R_{\geq 0} \to X$ where $\alpha$ isometrically embeds $\R_{\geq 0}$ onto the left ray in the graph rooted at the join-point, and similarly for $\alpha'$ and the right ray. It is clear that $|\ends(X)|=1$, and so $\text{end}(\alpha) = \text{end}(\alpha')$, however we claim that $[\alpha]^{\text{Crs}}_0 \neq [\alpha]^{\text{Crs}}_0$ and so that $|\pi_0^{\text{Crs}}(X)|>1$. For the sake of contradiction, suppose not. Then by Proposition \ref{discretealislip} there exists a proper $A$-Lipschitz map $\varphi: c_1([0,1])\cap \Z^2 \to X$ with $\varphi(0,h) = \alpha(h)$ and $\varphi(h,h) = \alpha'(h)$.
	
	For each $h \in \Z_{\geq 0}$ we can define a map $u_h:\{0,1,...,h\} \to X$ by the following composite.
	\begin{center}
		\begin{tikzcd}
			\{0,1,...,h\} \arrow[rr, "u_h"]  \arrow[dr, hook, "\text{inc}_h" swap] & & X \\
			\ & c_1([0,1])\cap \Z^2 \arrow[ur, "\varphi" swap]& \ 
		\end{tikzcd}
	\end{center} 
	Then each $u_h$ is defines an $A$-path in $X$ by the Lipschitz property. The contradiction will rely on the `crossings' of steps by $u_h$ for each $h$, as such we need define what we mean by a crossing. Let $j\in \Z_{\geq 1}$ and define a \textbf{forward crossing} of step$(j)$ to be a subset
	\begin{align*}
		\{u_h(x),u_h(x+1),...,u_h(x')\} \subset \text{im}(u_h)
	\end{align*}
	such that
	\begin{enumerate}
		\item $u_h(x)\in $ OB$(\alpha(j),A) \cap $ step$(j)$, called the \textbf{start} of the crossing
		\item $u_h(x')\in $ OB$(\alpha'(j),A) \cap $ step$(j)$, called the \textbf{end} of the crossing, and 
		\item $u_h\big(\{x+1,...,x'-1\}\big)\in$ step$(j)\backslash\big(\text{OB}(\alpha(j),A) \cup \text{OB}(\alpha'(j),A)\big)$, 
	\end{enumerate}
	and denote it by $f_j$. Define a \textbf{backwards crossing} of step$(j)$ similarly just with the sets in $(1)$ and $(2)$ swapped, denoting it by $b_j$. Create a word cross$(u_h, x,x')$ by order of appearance of all crossings in $u_h(\{x,x+1,...,x'\})$ as $i$ increases in $x+i$. Formally declare thatfor each $j$ we have $(f_j)^{-1}=b_j$, and reduce cross$(u_h,x,x')$ in the usual word-theoretic sense and call the resulting word describing the `behaviour' of $u_h$ by behav$(u_h,x,x')$. Let cross$(u_h)\coloneq \text{cross}(u_h,0,h)$ and $\text{behav}(u_h) \coloneqq \text{behav}(u_h,0,h)$. Notice, for any order-preserving subdivision of $\{1,2,...,h\}$,
	\begin{align*}
		\{0,1,...,i_0=x_1\} & \cup \{x_1,x_1+1,...,x_1+i_1=x_2\}\cup ... \\
		& \hspace*{1cm} \cup \{x_{m-1},x_{m-1}+1,...,x_{m-1}+i_{m-1}=x_{m}\} \\
		& \hspace*{2cm} \cup \{x_m,x_m+1,...,x_m+i_m = h\} = \{0,...,h\}
	\end{align*}
	such that each for each $i$, $u_h(x_i) \in OB(\text{im}(\alpha),A)\cup OB(\text{im}(\alpha'),A)$, we have that
	\begin{equation*}
		\text{cross}(u_h,0,x_1)\text{cross}(u_h,x_1,x_2)...\text{cross}(u_h,x_{m-1},x_m)\text{cross}(u_h,x_m,h) = \text{cross}(u_h),
	\end{equation*}
	and
	\begin{equation*}
		\text{behav}(u_h,0,x_1)\text{behav}(u_h,x_1,x_2)...\text{behav}(u_h,x_{m-1},x_m)\text{behav}(u_h,x_m,h) = \text{behav}(u_h),
	\end{equation*}
	under concatenation.
	
	Now, by the properness of $\varphi$, there exists some $R>0$ such that for all $h>R$, we have
	\begin{align*}
		\text{im}(u_h)\cap \Bigg(\alpha[0,A]\cup\alpha'[0,A]\cup \bigcup_{j\leq A}\text{step}(j)\Bigg) = \emptyset,
	\end{align*}
	i.e., for $h>R$ the image of $u_h$ lies entirely above step$(A)$. Further, for $h>R$ the following facts are easily checked,
	\begin{enumerate}
		\item both cross$(u_h)$ and behav$(u_h)$ must not be the empty word,
		\item both cross$(u_h)$ and behav$(u_h)$ must alternate between forward and backwards crossings,
		\item both cross$(u_h)$ and behav$(u_h)$ must begin and end with a forward crossing, and
		\item neither cross$(u_h)$ nor behav$(u_h)$ can contain instances of $f_h$ or $b_h$.
	\end{enumerate}
	
	Fix some such $h>R$. The claim is that for each $h'\geq h$ we have behav$(u_{h'})=\text{behav}(u_h)$. We will do this by induction. The zero-case is straightforward, i.e., behav$(u_h)=\text{behav}(u_h)$, so consider $h+n$ for some $n\geq 1$, and assume $\text{behav}(h+n)=\text{behav}(h)$. We will show that
	$$\text{behav}(h+n+1)=\text{behav}(h+n) = \text{behav}(h).$$ 
	For this let 
	\begin{equation*}
		x_0 \coloneqq \min\big(x \in \{0,1,...,h+n\}\ |\ x \text{ defines the start of a crossing by either } u_{h+n} \text{ or } u_{h+n+1}\big),
	\end{equation*}
	and say $x_0$ is the start of a crossing over step$(j_0)$, i.e., the first crossing to occur in either $u_{h+n} \text{ or } u_{h+n+1}$, and let
	\begin{equation*}
		x_0' \coloneqq \min\big(x \in \{x_0,x_0+1,...,h+n\}\ |\ u_{h+n}(x) \not\in \text{step}(j_0) \text{ or } u_{h+n}(x) \not\in \text{step}(j_0)\big).
	\end{equation*}
	i.e., the instance in which either $u_{h+n} \text{ or } u_{h+n+1}$ leave step$(j_0)$ after this crossing. This could either be `into' the image of $\alpha$ or the image of $\alpha'$. Suppose first that $x_0'\in \text{CB}(\text{im}(\alpha),A)$. If cross$(u_{h+n},x_0,x_0')= \emptyset$ we must have cross$(u_{h+n+1},x_0,x_0')$ is of the form $f_{j_0}b_{j_0}f_{j_0}...b_{j_0}$ and so 
	$$\text{behav}(u_{h+n+1},x_0,x_0')=\text{behav}(u_{h+n},x_0,x_0')=\emptyset.$$
	
	So instead suppose cross$(u_{h+n},x_0,x_0')\neq \emptyset$. Then it must begin with an instance of $f_{j_0}$. Let 
	\begin{equation*}
		x_0''\coloneqq \max(x\in \{x_0,...,x'_0\}\ |\ u_{h+n}(x)\in \text{OB}(\alpha'(j),A)).
	\end{equation*}
	Then either the $A$-path $u_{h+n}(\{x_0'',...,x_0'\})$ or $u_{h+n}(\{x_0'',...,x_0'-1\})$ defines a backward crossing of step$(j_0)$. And so the first entry in cross$(u_{h+n},x_0,x_0')$ is an instance of $f_{j_0}$ and the final entry is an instance of $b_{j_0}$, and since cross$(u_{h+n})$ must alternate between forward and backwards crossings and all points in-between lie on step$(j_0)$, cross$(u_{h+n},x_0,x'_0)$ is of the form $f_{j_0}b_{j_0}f_{j_0}...b_{j_0}$ and so behav$(u_{h+n},x_0,x_0')$ is empty. A similar argument can be used to show that behav$(u_{h+n+1},x_0,x_0')$ is also empty.
	
	Now suppose instead $x_0'\in \text{CB}(\text{im}(\alpha'),A)$. Then both cross$(u_{h+n},x_0,x_0')$ and cross$(u_{h+n},x_0,x_0')$ are non-empty and begin with instances of $f_{j_0}$. Let
	\begin{equation*}
		x_0''\coloneqq \max(x\in \{x_0,...,x'_0\}\ |\ u_{h+n}(x)\in \text{OB}(\alpha(j),A)). 
	\end{equation*}
	Then either the $A$-path $u_{h+n}(\{x_0'',...,x_0'\})$ or $u_{h+n}(\{x_0'',...,x_0'-1\})$ defines a forward crossing of step$(j_0)$. And so the first entry in cross$(u_{h+n},x_0,x_0')$ is an instance of $f_{j_0}$ and the final entry is an instance of $f_{j_0}$, and since cross$(u_{h+n})$ must alternate between forward and backwards crossings and all points in-between lie on step$(j_0)$, cross$(u_{h+n},x_0,x'_0)$ is of the form $f_{j_0}b_{j_0}f_{j_0}...f_{j_0}$ and so behav$(u_{h+n},x_0,x_0') =f_{j_0}$. A similar argument can be used to show behav$(u_{h+n+1},x_0,x_0')=f_{j_0}$.
	
	Next, let
	\begin{equation*}
		x_1 \coloneqq \min(x \in \{x_0',x_0'+1,...,h+n\}\ |\ x \text{ defines the start of a crossing by either } u_{h+n} \text{ or } u_{h+n+1}),
	\end{equation*}
	say a crossing over step$(j_1)$, and let
	\begin{equation*}
		x_1' \coloneqq \min(x \in \{x_1,x_1+1,...,h+n\}\ |\ u_{h+n}(x) \not\in \text{step(j)} \text{ or } u_{h+n}(x) \not\in \text{step(j)}).
	\end{equation*}
	Then we may use a similar argument to above (flipping the crossings if the initial crossing is a backwards one), to show that either behav$(u_{h+n},x_1,x_1')$ and behav$(u_{h+n +1},x_1,x_1')$ are both empty, 
	\begin{equation*}
		\text{behav}(u_{h+n},x_1,x_1') = \text{behav}(u_{h+n},x_1,x_1') = f_{j_1},
	\end{equation*}
	or
	\begin{equation*}
		\text{behav}(u_{h+n},x_1,x_1') = \text{behav}(u_{h+n},x_1,x_1') = b_{j_1}.
	\end{equation*}
 	Notice, behav$(u_{h+n},x_0',x_1)$ and behav$(u_{h+n+1},x_0,x_1)$ are empty by definition.
 	
 	We then continue to repeat this argument, dividing up $\text{cross}(u_{h+n})$ and $\text{cross}(u_{h+n})$ into pieces distinctly separated by the mutual beginning and ending of crossings. This must terminate by the finiteness of the $A$-paths $u_{h+n}$ and $u_{h+n+1}$, say $m$ iterations. The extra point $u_{h+n+1}$ has to work with is not a problem since cross$(u_{h+n+1})$ cannot contain any instances of $f_{h+n+1}$ and so the image of this point cannot be part of a crossing. In doing this, we gain a subdivision
	\begin{equation*}
		\{0,x_0\}\cup \{x_0,x'_0\} \cup ... \cup \{x_{i-1}',x_i\} \cup \{x_i,x_i'\} \cup ... \cup \{x_m,x_m'\} \cup \{x_m',h+n+1\}
	\end{equation*}
	such that 
	\begin{enumerate}
		\item $\text{behav}\{u_{h+n+1},0,x_0\}=\text{behav}\{u_{h+n},0,x_0\}$,
		\item $\text{behav}\{u_{h+n+1},x_{i-1}',x_i\}=\text{behav}\{u_{h+n},x_{i-1}',x_i\}$,
		\item $\text{behav}\{u_{h+n+1},x_i,x_i'\}=\text{behav}\{u_{h+n},x_i,x_i'\}$, and
		\item $\text{behav}\{u_{h+n+1},x_m',h+n+1\}$ is empty,
	\end{enumerate}
	and hence, $\text{behav}(u_{h+n+1})=\text{behav}(u_{h+n})=\text{behav}(u_{h})$.
	
	Finally, fix some $j$ such that $f_j$ appears in $\text{behav}(u_h)$, and pick some $(x_{h'},h')\in c_1([0,1]^2)$ such that $\varphi(x_{h'},h')$ for each $h'\geq h$. Then $\{(x_{h'},h')\} \not\in \B_{c_1([0,1])^2}$ but $\varphi(\{(x_{h'},h')\}) \subset \text{step}(j) \in \B_{X}$, and so $\varphi$ is not proper, and we reach a contradiction.
\end{proof}

\begin{corollary}
	The natural surjection $\pi_0^{\text{Crs}}(-) \to \ends(-) : \textbf{PGeo}^{\text{Crs}} \twoheadrightarrow \textbf{Set}$ is not a natural injection.
\end{corollary}

\begin{proof}
	Straightforward from Theorem \ref{ends_neq_pi0}.
\end{proof}

The question then becomes \textit{what subcategories of $\textbf{PGeo}^{\text{Crs}}$ may we restrict to in order to gain a natural injection?} The remainder of this article is devoted to showing that such a suitable subcategory containing all locally finite geometric trees suffices, providing the beginnings of an answer to this question.


\section{Geodesics in Locally Finite Geometric Trees}

In this section, we include a short interlude to covering some results relating to the behaviour of geodesics in locally finite geometric trees. First, we must define what we mean by a locally finite geometric tree.

\begin{definition}
	A \textbf{geometric tree} $T$ is a contractible simplicial $1$-complex equipped with the canonical path metric where each edge is isometric to the unit interval. We call $T$ \textbf{locally finite} if there is a finite number of $1$-simplices attached to each $0$-simplex. A pair of $0$-simplices are said to be \textbf{adjacent} if they are connected by a $1$-simplex.
\end{definition}

We will refer to the $0$-simplices and $1$-simplices in the definition above as vertices and edges, respectively, denoting the set of vertices of $T$ by $\text{Vert}(T)$. 

\begin{lemma}[Lemma 1.26, \cite{roe2003lectures}]\label{concat_of_geos_is_geo}
	Let $T$ be a locally finite geometric tree, then $T$ is uniquely geodesic. Further, if $x_1,x_2,x_3$ are points (not necessarily vertices) in $T$ with $$\text{im}(u_{x_1,x_2})\cap\text{im}(u_{x_2,x_3})=x_2,$$
	then the map $u: [0,a_{x_1,x_2}+a_{x_2,x_3}] \to T$ defined by
	\begin{equation*}
		u(t)\coloneq \begin{cases}
			u_{x_1,x_2}(t) & 0 \leq t \leq a_{x_1,x_2},\\
			u_{x_1,x_2}(t - a_{x_1,x_2}) & a_{x_1,x_2} \leq t \leq a_{x_1,x_2}+a_{x_2,x_3},
		\end{cases}
	\end{equation*}
	is a (and hence the unique) geodesic from $x_1$ to $x_3$.
\end{lemma}

Two further facts that will also be useful are that locally finite geometric trees are proper, and the intersection of a pair of geodesics on a locally finite tree also defines a geodesic. The former follows from the tree being locally finite, and the latter from the facts that the intersection of closed subsets is closed.

\begin{lemma}\label{trees_six_lemma}
	Let $T$ be a locally geometric tree. Let $x_1,x_2,x_3,y_1,y_2,y_3$ be points (not necessarily vertices) in $T$, such that
	\begin{enumerate}
		\item $d_T(x_1,y_1),d_T(x_3,y_3)<R$ for some $R>0$,
		\item $x_2$ lies on the geodesic $u_{x_1,x_3}$,
		\item $y_2$ lies on the geodesic $u_{y_1,y_3}$,
		\item if $y_2 \neq y_1$ then $d_T(x_2,x_3) \leq d_T(y_2,y_3)$, and
		\item if $x_2 \neq x_1$ then $d_T(x_2,x_3) \geq d_T(y_2,y_3)$.
	\end{enumerate}
	Then $d_T(x_2,y_2)<2R$.
\end{lemma}

Proving the above Lemma is just a long but routine check of all possible configurations of the geodesics $u_{x_1,x_3}$, $u_{y_1,y_3}$, and $u_{x_2,y_2}$, and as such we have chosen to omit it.

\begin{proposition}\label{geos_are_sequences_of_verts}
	Let $T$ be a locally finite geometric tree. Then the set of geodesics with vertex endpoints is in one to one correspondence with the set of finite sequences of adjacent vertices where each vertex appears at most once. Further, let $v_1,...,v_n$ be vertices in $T$. Then
	\begin{equation*}
		\text{im}(u_{v_1,v_n}) \subseteq \bigcup_{i \in \textbf{n-1}} \text{im}(u_{v_i,v_{i+1}}).
	\end{equation*}
\end{proposition}

\begin{proof}
	For the first claim, define a map by sending every finite sequence of non-repeating adjacent vertices by $(v_1,...,v_n)\mapsto u_{v_1,v_n}$. This is shown to be surjective by an application of the intermediate value theorem, and injective by the uniqueness of geodesics and the fact that there are no cycles since $T$ is a tree.
	
	For the second claim, again use the intermediate value theorem to gain a sequence of adjacent vertices. Then by removing all occurrences of $...,v,v',v,...$ and $...v,v,...$ in any order continuously until we halt (which we will since there is a finite number of points), we will be left with a sequence of non-repeating (since there are no cycles) adjacent vertices, and hence the unique geodesic with underlying image.
\end{proof}

\begin{lemma}\label{lemma_geo_underlies_concat_of_geos_trees}
	Let $x_1,...,x_n$ a finite sequence of points (not necessarily vertices) in some locally finite geometric tree $T$. Then
	\begin{equation*}
		\text{im}(u_{x_1,x_n}) \subseteq \bigcup_{i \in \boldsymbol{n-1}} \text{im}(u_{x_i,x_{i+1}})
	\end{equation*}
\end{lemma}

\begin{proof}
	If all $x_i$ are vertices we are done by above. Further if all $x_i$ lie on a single edge this is trivial, so suppose otherwise. Consider the concatenation of these geodesics
	\begin{equation*}
		u:\left[0, \sum_i a_{x_i,x_{i+1}} \right] \to T
	\end{equation*}
	defined by
	\begin{equation*}
		u(t) \coloneq \begin{cases}
			u_{x_1,x_2}(t) & 0 \leq t \leq a_{x_1,x_2},\\
			\vdots & \vdots\\
			u_{x_j,x_{j+1}}\left(t-\sum_{i=1}^{j-1}a_{x_i,x_{i+1}} \right) & \sum_{i=1}^{j-1}a_{x_i,x_{i+1}} \leq t \leq \sum_{i=1}^{j}a_{x_i,x_{i+1}},\\
			\vdots & \vdots\\
			u_{x_{n-1},x_{n}}\left(t-\sum_{i=1}^{n-2}a_{x_i,x_{i+1}} \right) & \sum_{i=1}^{n-2}a_{x_i,x_{i+1}} \leq t \leq \sum_{i=0}^{n-1} a_{x_i,x_{i+1}}.
		\end{cases}
	\end{equation*}
	
	Let $v$ be the `first vertex hit' by $u$. That is $v \in Vert(T)$ such that $u(t)=v$ for some $t$ such that for all $t'$ with $u(t') \in \text{Vert}(T)$ we have $t\leq t'$. Similarly define $v'$ to be the `last vertex hit' by $u$. We plan to define a finite sequence of adjacent vertices whose image contains the image of $u$. These exist by the isometry of geodesics.
	
	If $x_1,x_n \in \text{Vert}(T)$ set $v_1 \coloneq x_1 = v$ and $v_n \coloneq x_n = v'$, respectively. If not, i.e., $x_1 \not\in \text{Vert}(T)$, set $v_1$ to be the vertex adjacent to $v$ with the edge connecting them containing $x_1$. Define $v_n$ similarly if $x_n \not\in \text{Vert}(T)$. Call these edges $e_1$ and $e_n$ respectively. 
	
	Further, if $x_j$ lies on $e_1$ such for all $i \leq j$ we have $x_i$ also lies on $e_1$ then define $v_j$ to be $v_1$. Similarly, if $x_j$ lies on $e_n$ and for all $i \geq j$ we have $x_i$ also lies on $e_n$, then define $v_j$ to be $v_n$.
	
	Next, for each other $j \in \{0,...,n\}$ not yet considered, let $v_j \in \text{Vert}(T)$ such that $v_j = u(t)$ for some $t \leq \sum_{i=0}^{j-1}a_{x_i,x_{i+1}}$ such that for all $t<t'\leq \sum_{i=0}^{j-1}a_{x_i,x_{i+1}}$ we have $u(t') \not\in \text{Vert}(T)$, i.e., the `most recent vertex hit' at or before $x_j$. An example setup where $x_1 \not\in \text{Vert}(T)$ and $x_4$ is the first point not lying on $e_1$ is depicted below.
	
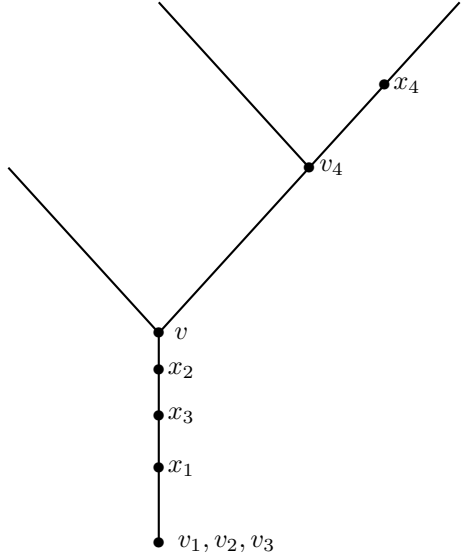
\begin{center}
		\begin{tikzpicture}
			\draw[thick]  (0,0) -- (0,2.8);
			\draw[thick]  (0,2.8) -- (-2,5);
			\draw[thick]  (0,2.8) -- (2,5);
			\draw[thick]  (2,5) -- (0,7.2);
			\draw[thick]  (2,5) -- (4,7.2);
			
			\node at (0.3,1) {$x_1$};
			\node at (0.3,2.3) {$x_2$};
			\node at (0.3,1.7) {$x_3$};
			\node at (0,1) {$\bullet$};
			\node at (0,2.3) {$\bullet$};
			\node at (0,1.7) {$\bullet$};
			
			\node at (3,6.1) {$\bullet$};
			\node at (3.3,6.1) {$x_4$};
			
			\node at (0,0) {$\bullet$};
			\node at (0.9,0) {$v_1,v_2,v_3$};
			
			\node at (0,2.8) {$\bullet$};
			\node at (0.3,2.8) {$v$};
			
			\node at (2,5) {$\bullet$};
			\node at (2.3,5) {$v_4$};
		\end{tikzpicture}
	\end{center}
	By Proposition \ref{geos_are_sequences_of_verts}, we then have that 
	\begin{equation*}
		\text{im}\left(u_{v_1,v_n}\right) \subseteq \bigcup_{i \in \boldsymbol{n-1}} \text{im}(u_{v_i,v_{i+1}}).
	\end{equation*}
	Then if $x_1,x_n\in \text{Vert}(T)$ we have
	\begin{equation*}
		\text{im}\left(u_{x_1,x_n}\right) = \text{im}\left(u_{v_1,v_n}\right) \subseteq \bigcup_{i \in \boldsymbol{n-1}} \text{im}(u_{v_i,v_{i+1}}),
	\end{equation*}
	if $x_1\in \text{Vert}(T)$ but $x_n \not\in \text{Vert}(T)$ then 
	\begin{equation*}
		\text{im}\left(u_{x_1,x_n}\right) = \text{im}\left(u_{v_1,v_n}\right)\backslash \big((im(u)^\text{c}\cap e_n)\big) \subseteq \bigcup_{i \in \boldsymbol{n-1}} \text{im}(u_{v_i,v_{i+1}})\backslash \big((im(u)^\text{c}\cap e_n)\big),
	\end{equation*}
	similarly if $x_n\in \text{Vert}(T)$ but $x_1 \not\in \text{Vert}(T)$ then 
	\begin{equation*}
		\text{im}\left(u_{x_1,x_n}\right) = \text{im}\left(u_{v_1,v_n}\right)\backslash \big((im(u)^\text{c}\cap e_1)\big) \subseteq \bigcup_{i \in \boldsymbol{n-1}} \text{im}(u_{v_i,v_{i+1}})\backslash \big((im(u)^\text{c}\cap e_1)\big),
	\end{equation*}
	and if $x_1,x_n \not\in \text{Vert}(T)$ then 
	\begin{align*}
		\text{im}\left(u_{x_1,x_n}\right) & = \text{im}\left(u_{v_1,v_n}\right)\backslash \big((im(u)^\text{c}\cap e_n) \cup (im(u)^\text{c}\cap e_n) \big) \\
		& \subseteq \bigcup_{i \in \boldsymbol{n-1}} \text{im}(u_{v_i,v_{i+1}})\backslash \big((im(u)^\text{c}\cap e_1) \cup (im(u)^\text{c}\cap e_n) \big).
	\end{align*}
	Then, we claim in each case above we have
	\begin{align*}
		\bigcup_{i \in \boldsymbol{n-1}} \text{im}(u_{v_i,v_{i+1}}) & \subseteq \bigcup_{i \in \boldsymbol{n-1}} \text{im}(u_{x_i,x_{i+1}}),\\
		\bigcup_{i \in \boldsymbol{n-1}} \text{im}(u_{v_i,v_{i+1}})\backslash \big((im(u)^\text{c}\cap e_n)\big) & \subseteq \bigcup_{i \in \boldsymbol{n-1}} \text{im}(u_{x_i,x_{i+1}}),\\
		\bigcup_{i \in \boldsymbol{n-1}} \text{im}(u_{v_i,v_{i+1}})\backslash \big((im(u)^\text{c}\cap e_1)\big) & \subseteq \bigcup_{i \in \boldsymbol{n-1}} \text{im}(u_{x_i,x_{i+1}}), \text{ and }\\
		\bigcup_{i \in \boldsymbol{n-1}} \text{im}(u_{v_i,v_{i+1}})\backslash \big((im(u)^\text{c}\cap e_1) \cup (im(u)^\text{c}\cap e_n) \big) & \subseteq \bigcup_{i \in \boldsymbol{n-1}} \text{im}(u_{x_i,x_{i+1}}),
	\end{align*}
	respectively. We will show the latter, then the rest are similar. We will do this by showing for each $j\in\{1,...,n-1\}$ we have
	\begin{equation*}
		\text{im}(u_{v_j,v_{j+1}})\backslash \big((im(u)^\text{c}\cap e_1) \cup (im(u)^\text{c}\cap e_n) \big) \subseteq \bigcup_{i \in \boldsymbol{n-1}} \text{im}(u_{x_i,x_{i+1}}),
	\end{equation*}
	assuming $x_j$ and $x_{j+1}$ do not lie on the same edge (otherwise this is trivial).
	
	First suppose that there exists $t< \sum_{i}^{j-1}a_{x_{i},x_{i+1}}$ such that $u(t)\in \text{Vert}(T)$ or $u(t)$ lies on a different edge, and, that there exists some $t'> \sum_{i}^{j}a_{x_{j},x_{j+1}}$ such that $u(t')\in \text{Vert}(T)$ or $u(t')$ lies on a different edge to $x_j$. Then the relevant geodesics can be depicted as follows,
	
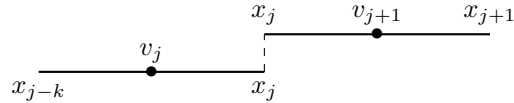
\begin{center}
		\begin{tikzpicture}
			\draw[thick] (0,0) -- (3,0);
			\draw[thick] (3,0.5) -- (6,0.5);
			\draw[dashed] (3,0) -- (3,0.5);
			
			\node at (0,-0.25) {$x_{j-k}$};
			\node at (3,-0.25) {$x_{j}$};
			\node at (3,0.75) {$x_{j}$};
			\node at (6,0.75) {$x_{j+1}$};
			
			\node at (1.5,0) {$\bullet$}; 
			\node at (4.5,0.5) {$\bullet$};
			
			\node at (1.5,0.25) {$v_j$};
			\node at (4.5,0.75) {$v_{j+1}$};
		\end{tikzpicture}
	\end{center}
	for some $k$, and $v_{j+1}$ lies somewhere on this diagram (on the left line if both $x_j$ and $x_{j+1}$ lie on the same edge, and on the right line otherwise). Then, this reduces to the following geodesic diagram,
	
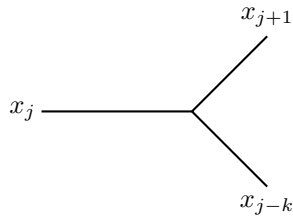
\begin{center}
		\begin{tikzpicture}
			\draw[thick] (-2,0) -- (0,0);
			\draw[thick] (1,1) -- (0,0) -- (1,-1);
			
			\node at (-2.25,0) {$x_{j}$};
			\node at (1,1.25) {$x_{j+1}$};
			\node at (1,-1.25) {$x_{j-k}$};
		\end{tikzpicture}
	\end{center}
	where $v_j$ and $v_{j+1}$ lie somewhere on this diagram, and we are allowing for constant geodesics. Hence, the image of the unique geodesic $u_{v_j,v_{j+1}}$ lies on this diagram in the image of $u$.
	
	Next, suppose there exists no such $t< \sum_{i}^{j-1}a_{x_{j},x_{j+1}}$ and $t'> \sum_{i}^{j}a_{x_{j},x_{j+1}}$ such that $u(t)$ or $u(t')$ lie on vertices or $u(t)$ lies on a different edge to $x_j$ or $u(t')$ lies on a different edge to $x_{j+1}$. Then $v_j = v_1$ and $v_{j+1}=v_n$, and $x_j$ lies on $e_1$ and $x_{k+1}$ lies on $e_n$. Then by definition, the following concatenation of geodesics is itself a geodesic,
	
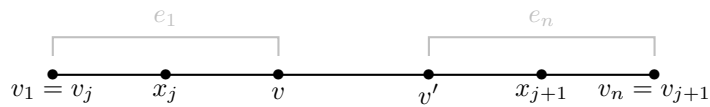
\begin{center}
		\begin{tikzpicture}
			\draw[thick] (0,0) -- (8,0);
			
			\node at (0,-0.25) {$v_1=v_j$};
			\node at (8,-0.25) {$v_n=v_{j+1}$};
			\node at (0,0) {$\bullet$};
			\node at (8,0) {$\bullet$};
			
			\node at (1.5,0) {$\bullet$}; 
			\node at (6.5,0) {$\bullet$};
			\node at (1.5,-0.25) {$x_{j}$};
			\node at (6.5,-0.25) {$x_{j+1}$};
			
			\node at (3,0) {$\bullet$}; 
			\node at (5,0) {$\bullet$};
			\node at (3,-0.25) {$v$};
			\node at (5,-0.25) {$v'$};
			
			\draw[thick,opacity=0.25] (0,0.25) -- (0,0.5) -- (3,0.5) -- (3,0.25);
			\node[opacity = 0.25] at (1.5,0.75) {$e_1$}; 
			
			\draw[thick,opacity=0.25] (5,0.25) -- (5,0.5) -- (8,0.5) -- (8,0.25);
			\node[opacity = 0.25] at (6.5,0.75) {$e_n$}; 
		\end{tikzpicture}
	\end{center}
	i.e., $\text{im}(u_{v_j,v_{j+1}})$ is just the union of $e_1,e_n$ and $\text{im}(u_{x_j,x_{j+1}})$. Then, we get
	\begin{equation*}
		\text{im}(u_{v_j,v_{j+1}})\backslash \big((im(u)^\text{c}\cap e_1) \cup (im(u)^\text{c}\cap e_n) \big) \subseteq \bigcup_{i \in \boldsymbol{n-1}} \text{im}(u_{x_i,x_{i+1}}).
	\end{equation*}
	The other two cases are similar to this one, that is the case when there exists $t< \sum_{i}^{j-1}a_{x_{j},x_{j+1}}$ with $u(t)\in \text{Vert}(T)$ or lying on a different edge to $x_{j}$ but no $t'> \sum_{i}^{j}a_{x_{j},x_{j+1}}$ with $u(t')\in \text{Vert}(T)$ or lying on a different edge to $x_{j+1}$, and the symmetrical case where there exists $t'> \sum_{i}^{j}a_{x_{j},x_{j+1}}$ with $u(t')\in \text{Vert}(T)$ or lying on a different edge to $x_{j+1}$ but no $t< \sum_{i}^{j-1}a_{x_{j},x_{j+1}}$ such that $u(t)\in \text{Vert}(T)$ or lying on a different edge to $x_{j}$. Hence, the result holds.
\end{proof}


\section{Coarse Homotopy and Ends of Locally Finite Geometric Trees}

Now, we will show that a suitable subcategory of locally finite geometric trees satisfies the questions posed at the end of Section \ref{section_7}. First, the following common fact in category theory will be useful.

\begin{lemma}[Lemma 1.3.11, \cite{leinster2014basic}]\label{leinster_nat_iso}
	Let $\A$ and $\B$ be categories, $F,G:A\to B$ be functors and $\eta : F \Rightarrow G$ a natural transformation. Then $\eta$ is a natural isomorphism if and only if each for each object $A\in \A$ the component $\eta_A$ is an isomorphism.
\end{lemma}

The above lemma tells us that if we wish to show the natural surjection in Theorem \ref{ends_natural_surjection} is a natural isomorphism when restricted to some subcategory of $\textbf{PGeo}^{\text{Crs}}$, we need only show the map $\eta_X : \pi_0^{\text{Crs}}(X) \to \ends(X)$ is an injection for each such space $X$ in the chosen subcategory. That is, for any pair of appropriate rays $\alpha,\alpha':\R_{\geq 0} \to X$, we have that $\text{end}(\alpha^*) = \text{end}((\alpha')^*)$ implies $\pi_0^{\text{Crs}}(\alpha) = \pi_0^{\text{Crs}}(\alpha')$. In this section, we do this for locally finite geometric trees.

\begin{definition}
	Let $P$ be a poset. A \textbf{chain} $C$ in $P$ is a totally ordered subset of $P$. Call a chain $C$ \textbf{order-convex} if for all $a,b \in C$ and $d\in P$ such that $a\leq d \leq b$ we have $d\in C$. 
\end{definition}

The above definitions of chains and order-convexivity are common ones (see \cite{semenova2004sublattices} for instance).

\begin{proposition}\label{poset_tree}
	Let $T$ be a locally finite geometric tree and let $v_0 \in \text{Vert}(T)$. Then there is a canonical poset structure on $\text{Vert}(T)$ where $v \leq v'$ if and only if $d_T(v_0,v) \leq d_T(v_0,v')$. Further, this poset structure has all meets.
	
	Also, $v_0$-rooted geodesic rays in $T$ are in one to one correspondence with infinite order-convex chains in this poset structure with least element $v_0$.
\end{proposition}

\begin{proof}
	The first claim is relatively straightforward, just a quick check of the poset axioms. For the second claim, using the intermediate value theorem and the absence of cycles we see that each $v_0$-rooted geodesic ray contains in its image a unique infinite order-convex chain with least element $v_0$. We then define the correspondence by mapping each $v_0$-rooted geodesic to this unique underlying change. 
	
	To see this is injective, it is enough to notice that if two $v_0$-rooted geodesic rays contain the same underling chain, and then by uniqueness of geodesics, they must be the same geodesic ray. To see this is a surjective mapping, we simply take an infinite order-convex chain with least element $v_0$ and construct a $v_0$-rooted geodesic by joining each vertex with an isometric mapping of the unit interval.
\end{proof}

\begin{proposition}\label{geos_in_trees_diverge}
	Let $T$ be a locally finite geometric tree with root $v_0$ and $r,r':\R_{\geq 0} \to T$ be $v_0$-rooted geodesic rays. If there exists a vertex $v$ in $T$ such that $v \in \text{im}(r)\backslash \text{im}(r')$ then there exist some $h\geq 0$ such that $r(h)=r'(h)\in \text{Vert}(T)$ and for all $h'>h$ we have $r(h')\neq r'(h')$.
\end{proposition}

\begin{proof}
	Suppose, for the sake of contradiction, that there exists some vertex $v\in im(r)\backslash im(r')$, but for all $h\geq 0$ with $r(h)=r'(h)\in \text{Vert}(T)$ there exists some $h'>h$ such that $r(h') = r'(h')$. If $r(h)=r'(h)$ for any $h\geq 0$, then $r(h')=r'(h')$ for all $h' \leq h$. This follows from the uniqueness of geodesics. If $h\in \Z_{\geq 0}$ is the unique integer with $r(h) = v$, choose any $h' >h$, and $v\in im(r)$, contradicting our assumption.
\end{proof}

\begin{proposition}
	Let $r,r': \R_{\geq 0} \to T$ be geodesic rays rooted at $v_0$. Then $\text{end}(r)= \text{end}(r')$ if and only if $r=r'$.
\end{proposition}

\begin{proof}
	The backwards implication is obvious, so for the forward implication suppose $r\neq r'$ but $\text{end}(r)= \text{end}(r')$. Then there exist some $h>0$ such that $r(h)=r'(h)\in \text{Vert}(T)$ and for all $h'>h$ we have $r(h')\neq r'(h')$. By assumption, there is a continuous path connecting $r|_{[h,\infty)}$ and $r'|_{[h',\infty)}$ in $T \backslash CB(v_0,d_T(v_0,r(t)))$ for some $h'$. But then we necessarily have a cycle in $T$ and we reach a contradiction.
\end{proof}

We also have the following.

\begin{proposition}\label{geo_pi0_iff_equal}
	Let $T$ be a locally finite geometric tree, $v_0 \in \text{Vert}(T)$ and $r,r':\R_{\geq 0} \to X$ be $v_0$-rooted geodesic rays. Then $[r]^{\text{Crs}}_0=[r']^{\text{Crs}}_0$ if and only if $r=r'$.
\end{proposition}

\begin{proof}
	Again, the backwards implication is straightforward, so for the forward implication suppose $r\neq r'$ but $[r]^{\text{Crs}}_0=[r']^{\text{Crs}}_0$. Then there exist some $h>0$ such that $r(h)=r'(h)\in \text{Vert}(T)$ and for all $h'>h $ we have $r(h')\neq r'(h')$. Then there is a coarse $1$-path $\varphi:r\myrightsquigarrow r'$ such that $\varphi|_{c_1([0,1])\cap \Z^2}$ is proper $A$-Lipschitz for some $A$. Notice, by Lemma \ref{concat_of_geos_is_geo}, for each $h'>h$ we have $r(h)\in \text{im}(u_{r(h'),r'(h')})$. Then, by Lemma \ref{lemma_geo_underlies_concat_of_geos_trees} we have
	\begin{equation*}
		\text{im}(u_{r(h'),r'(h')}) \subseteq \cup_{i\in \boldsymbol{h'-1}}\text{im}(u_{\varphi(i,h'),\varphi(i+1,h')}).
	\end{equation*}
	for each $h'>h$. Then, for all $h'>h$ there exists $(h't,h') \in c_1([0,1])\cap \Z^2$ such that we have $\varphi(h't,h') \in CB(r(h),A)$ and $\varphi$ is not proper, and we reach a contradiction.
\end{proof}

With the above two propositions, and the fact that every end of a proper geodesic space can be represented by some rooted geodesic ray for any choice of root, we need only show that coarse path component may be represented by some rooted geodesic ray, again for any choice of root.

\begin{proposition}
	Let $T$ be a locally finite geometric tree, $v_0\in \text{\text{Vert}(T)}$, and $\alpha^*:\R_{\geq 0} \to T$ be a $v_0$-rooted proper Lipschitz ray in $T$ constructed with geodesics as in Lemma \ref{every_crs_ray_close_to_plip_ray}, for some coarse ray $\alpha: \R_{\geq 0} \to X$. Then there exists a unique $v_0$-rooted geodesic ray $r:\R_{\geq 0} \to T$ such that $\text{im}(r)\subseteq \text{im}(\alpha^*)$.
\end{proposition}

\begin{proof}
	Equip $T$ with a the canonical poset structure based at $v_0$. First, we will show existence. By Proposition \ref{poset_tree} it suffices to show there exists is infinite order-convex chain with least element $v_0$ contained in the image of $\alpha^*$. For the sake of contradiction, suppose no such sequence exists. Then there exists some $x \in \text{Vert}(T)$ such that 
	\begin{equation*}
		d_T(x,v_0)=\text{max}\{d_T(x',v_0)\ |\ x\in \text{Vert}(T)\cap\text{im}(\alpha) \}.
	\end{equation*}
	As such, $\text{im}(\alpha^*) \subseteq \text{CB}(v_0,d_T(x,v_0)+1)$ and we reach a contradiction since $\alpha^*$ is proper. Order convexity follows from the intermediate value theorem.
	
	Next, we will show uniqueness. Suppose there is a second $v_0$-rooted geodesic ray $r'$ in $T$ with $\text{im}(r')\subseteq \text{im}(\alpha)$ then by Proposition \ref{geos_in_trees_diverge} there exists some $z \in \Z_{\geq 0}$ such $x \coloneqq r(z) = r'(z)$ but $r(k)\neq r'(k)$ for any $z<k\in \R_{\geq 0}$. By the properness of $\alpha^*$, there exists some $R>0$ such that $\alpha^*[R,\infty) \cap \{x\}=\emptyset$. Therefore, either $\text{im}(r)\subseteq \alpha^*[0,R)$ or $\text{im}(r')\subseteq \alpha^*[0,R)$. In either case, we reach a contradiction by the properness of geodesic rays.
\end{proof}

\begin{proposition}\label{tree_crs_to_geo_ray}
	Let $T$ be a locally geometric finite tree, $v_0\in \text{Vert}(T)$ and $\alpha: \R_{\geq 0} \to T$ be some coarse ray. Then there is a coarse $1$-path from $\alpha$ to some $v_0$-geodesic ray in $T$.
\end{proposition}

\begin{proof}
	By Lemma \ref{every_crs_ray_close_to_plip_ray} it suffices to show $\alpha^*$ (which is proper $A$-Lipschitz for some $A$) is coarse homotopic to some $v_0$-geodesic ray in $T$. Specifically, we will show that there is a coarse $1$-path from $\alpha^*$ to its unique underlying $v_0$-rooted geodesic ray $r:\R_{\geq 0} \to T$. We claim $1$-path $\varphi: \alpha^* \myrightsquigarrow r$ defined as follows, suffices.
	\begin{align*}
		\varphi(ht,h)\coloneqq \begin{cases}
			u_{r(h),\alpha^*(h)}\big(a_{r(h),\alpha^*(h)} - (A+1)ht\big) & ht \leq \frac{a_{r(h),\alpha^*(h)}}{A+1},\\
			r(h) & ht \geq \frac{a_{r(h),\alpha^*(h)}}{A+1}.
		\end{cases}
	\end{align*}
	Note, this makes sense since for all $h\in \R_{\geq 0}$, we have
	\begin{align*}
		a_{r(h),\alpha^*(h)} = d_T(\alpha^*(h),r(h)) \leq d_T(\alpha^*(h),\alpha^*(0)) + d_T(r(0),r(h)) \leq Ah + h = (A+1)h.
	\end{align*}
	To see $\varphi$ is proper, let $B$ be bounded in $T$. Notice, for each $h$ there exists some $h'$ such that $\alpha^*(h')=r(h)$. In particular, $\text{im}(u_{r(h),\alpha^*(h)})$ is contained in the images of the sequence of geodesic defined by $\alpha^*|_{[h,h']}$ if $h'\geq h$ or $\alpha^*|_{[h',h]}$ if $h\geq h'$ by Lemma \ref{lemma_geo_underlies_concat_of_geos_trees}. And so, for all bounded $B \subseteq T$, we have
	\begin{equation*}
		\varphi^{-1}[B] \subseteq (\alpha^*)^{-1}[B] \in \B_{c_1([0,1])}
	\end{equation*}
	and $\varphi$ is proper.
	
	To see $\varphi$ is controlled we may employ Proposition \ref{dimatatime}. First suppose $(ht,h),(ht,h') \in c_1([0,1])$ such that $d_{c_1([0,1])}((ht,h),(ht,h'))<R$ for some $R>0$. Then $|h-h'|<R$ and further, we have that $d_T(\alpha(h),\alpha'(h))<AR$. Then set
	$x_3 = \alpha^*(h),y_3 = \alpha^*(h')$, $x_1 = r(h)$, $y_1 = r(')$, and
	\begin{align*}
		x_2 & = u_{r(h),\alpha^*(h)}\big(a_{r(h),\alpha^*(h)} - (A+1)ht\big), \text{ and }\\
		y_2 & = u_{r(h'),\alpha^*(h')}\big(a_{r(h'),\alpha^*(h')} - (A+1)ht\big)
	\end{align*}
	then apply Lemma \ref{trees_six_lemma} to get
	\begin{equation*}
		d_T(\varphi(ht,h),\varphi(ht,h'))<2R,
	\end{equation*}
	and thus $\varphi \times \varphi$ preserves $V_R$.
	
	Now consider some $(ht,h),(ht',h) \in c_1([0,1])$ such that $d_{c_1([0,1])}((ht,h),(ht',h))<R$. Then, without loss of generality we have three cases to consider. First, suppose $ht,ht' \leq \frac{a_{r(h),\alpha^*(h)}}{A + 1}$, then
	\begin{align*}
		d_T\big(\varphi(ht,h),\varphi(ht',h)\big) & = d_T\big(u_{r(h),\alpha^*(h)}(a_{r(h),\alpha^*(h)}-(A+1)ht),\\
		& \hspace*{3cm} u_{r(h),\alpha^*(h)}(a_{r(h),\alpha^*(h)}-(A+1)ht')\big)\\
		& = d_{\R_{\geq 0}}\big((A+1)ht,(A+1)ht'\big)<(A+1)R.
	\end{align*}
	Next suppose $ht \leq \frac{a_{r(h),\alpha^*(h)}}{(A+1)}$ and $ht' \geq \frac{a_{r(h),\alpha^*(h)}}{(A+1)}$, then
	\begin{align*}
		d_T\big(\varphi(ht,h),\varphi(ht',h)\big) & = d_T\big(u_{r(h),\alpha^*(h)}(a_{r(h),\alpha^*(h)}-(A+1)ht),\alpha^*(h)\big)\\
		& \leq d_{\R_{\geq 0}}((A+1)ht,(A+1)ht')<(A+1)R.
	\end{align*}
	Finally suppose $ht,ht' \geq \frac{a_{r(h),\alpha^*(h)}}{(A+1)}$, then
	\begin{align*}
		d_T\big(\varphi(ht,h),\varphi(ht',h)\big) = d_T\big(r(h),r(h))\big)=0 <(A+1)R,
	\end{align*}
	and thus $\varphi\times \varphi$ preserves $U_R$, and hence by Proposition \ref{dimatatime}, $\varphi$ is controlled and we are done.
\end{proof}

Let $\textbf{lfGTree}^\text{Crs}$ be the category of locally finite geometric trees and coarse maps.

\begin{theorem}\label{tree_ends_pi0_thm}
	There exists a natural isomorphism 
	$$\ends(-)|_{\textbf{lfGTree}^{\text{Crs}}} \cong \pi_0^{\text{Crs}}(-)|_{\textbf{lfGTree}^{\text{Crs}}} : \textbf{lfGTree}^{\text{Crs}} \to \textbf{Set}.$$
\end{theorem}

\begin{proof}
	By Proposition \ref{tree_crs_to_geo_ray} each $\eta_T: \pi_0^{\text{Crs}}(T) \to \ends(T)$ in Theorem \ref{ends_natural_surjection} is an injection when restricted to any locally finite geometric tree $T$, and hence a natural isomorphism when restricted to the category $\textbf{lfGTree}^{\text{Crs}}$ by Lemma \ref{leinster_nat_iso}.
\end{proof}

In particular, the above Theorem can be extended to any subcategory of $\textbf{PGeo}^{\text{Crs}}$ whose objects have the coarse type (further coarse homotopy type) of a locally finite geometric tree.

Now, recall that a group is virtually \textit{blank} if it contains a subgroup of finite index with that property. For example, a group is \textbf{virtually free} if it contains the free group $F_n$, for some $n$, as a subgroup of finite index. Further, a finitely generated group is quasi-isometric (and hence coarse equivalent) to all of its finite index subgroups (see Corollary 8.47 of \cite{dructu2018geometric} for example). Then, we get the following corollary of Theorem \ref{tree_ends_pi0_thm}.

\begin{corollary}
	Let $\textbf{fgVirtFrGrp}^{\text{Crs}}$ be the category of Cayley graphs of finitely-generated virtually-free groups and coarse maps (seen as a subcategory of $\textbf{PGeo}^{\text{Crs}}$). Then, there exists a natural isomorphism $$\mathcal{E}\text{nds}(-)|_{\textbf{fgVirtFrGrp}^{\text{Crs}}} \cong \pi_0^{\text{Crs}}(-)|_{\textbf{fgVirtFrGrp}^{\text{Crs}}} : \textbf{fgVirtFrGrp}^{\text{Crs}} \to \textbf{Set}.$$
\end{corollary}

\begin{proof}
	If $\Gamma$ is a finitely generated virtually free group, then it contains $F_n$ as a subgroup of finite index for some $n\in \Z_{\geq 1}$. Then $\Gamma \overset{\text{Crs}}{\simeq} F_n$ and $\pi_0^{\text{Crs}}(\Gamma) \cong \pi_0^{\text{Crs}}(F_n) \cong \ends(F_n) \cong \ends(\Gamma)$, since finitely generated free groups may be viewed as locally finite geometric trees. The result then follows from Theorem \ref{tree_ends_pi0_thm}.
\end{proof}

We have begun to answer the question of what suitable subcategory we may restrict to in order for the natural surjection $\pi_0^{\text{Crs}}(-) \twoheadrightarrow \ends(-)$ seen in Theorem \ref{ends_natural_surjection} to become a natural isomorphism. There is, however, more work to be done.

A key difference between how the constructions of ends and coarse path components identify rays is that, when computing the set of ends, we do not care about the `area' of the holes a space may posses. However, we saw in Theorem \ref{ends_neq_pi0} that, when computing $\pi_0^{\text{Crs}}$, holes with increasing diameter have an effect on the calculation. The problem appears when we begin to consider how one defines the `size' of a hole. 

A place to start could be the class of finitely-presented groups where the length of the relations should bound the area of the holes in the space. There are notions of areas that exist in the literature that may be useful here (for example that of the algebraic area of a word in Definition $7.93$ of \cite{dructu2018geometric}). More sophisticated techniques would likely be needed than those present in this paper, where we took a relatively brute-force approach in order to extend Theorem \ref{tree_ends_pi0_thm} to a more general class of spaces.

\subsubsection*{Acknowledgements}
The work presented here makes up part of a doctoral thesis. As such, many thanks go to Paul Mitchener as the supervisor of this project. His guidance, patience, and encouragement were invaluable for the completion of this work. Further, the author would like to thank Alex Corner, for his feedback on an early draft of this article.


\bibliographystyle{alpha}
\bibliography{ref_ends.bib}

\end{document}